\documentclass[11pt]{article}

\usepackage{amsmath,amssymb,amstext,amsfonts}
\usepackage{amsthm}
\usepackage{enumerate}
\usepackage{graphicx}
\usepackage{subcaption}
\usepackage{float}
\usepackage[paperwidth=8in, paperheight=11in, margin=.875in]{geometry}
\usepackage{listings}
\usepackage{xcolor}
\usepackage{subcaption}

\usepackage{graphicx}
\newtheorem{theorem}{Theorem}[section]
\newtheorem{lemma}[theorem]{Lemma}

\newtheorem{definition}{Definition}[section]
\newtheorem*{remark}{Remark}
\usepackage{dsfont}

\newcommand{\xtu}{x_t^u}

\newcommand{\R}{\mathbb{R}}

\newcommand{\ff}{\Phi}

\newcommand{\cpd}{\text{strictly conditionally positive definite}}

\newcommand{\pindt}{\hspace{\parindent}}
\newcommand{\Dt}{\Delta t}
\newcommand{\Dx}{\Delta x}
\newcommand{\hEkn}{\hat{\mathbb{E}}_{t_n}^{x_k}}
\newcommand{\Exn}{\mathbb{E}_{t_n}^{x}}

\newcommand{\hE}{\hat{\mathbb{E}}}
\newcommand{\Or}{\mathcal{O}}

\newcommand{\bp}{\boldsymbol{p}}
\newcommand{\bq}{\boldsymbol{q}}
\newcommand{\bR}{\boldsymbol{R}}
\newcommand{\bmu}{\boldsymbol{\mu}}
\newcommand{\bnu}{\boldsymbol{\nu}}

\newcommand{\oo}{\mathcal{O}}

\newcommand{\E}{\mathbb{E}}

\date{}

\begin{document}

\title{Meshfree Approximation for Stochastic Optimal Control Problems}
\author{
Hui Sun,\thanks{Department of Mathematics, Florida State University, Tallahassee, Florida}  
\ Feng Bao \thanks{Department of Mathematics, Florida State University, Tallahassee, Florida, (bao@math.fsu.edu)}}
\maketitle
\begin{abstract}
In this work, we study the gradient projection method for solving a class of stochastic control problems by using a mesh free approximation approach to implement spatial dimension approximation. Our main contribution is to extend the existing gradient projection method to moderate high-dimensional space. The moving least square method and the general radial basis function interpolation method are introduced as showcase methods to demonstrate our computational framework, and rigorous numerical analysis is provided to prove the convergence of our meshfree approximation approach. We also present several numerical experiments to validate the theoretical results of our approach and demonstrate the performance meshfree approximation in solving stochastic optimal control problems.
\end{abstract}

\textbf{keywords:} stochastic optimal control, maximum principle, backward stochastic differential equations, meshfree approximation

\section{Introduction}
The stochastic optimal control problem is a very important research topic in both mathematical and engineering communities. There is a large number of literatures contributing theoretical foundations to the optimal control theories (\cite{wfms}, \cite{ntouzi}, \cite{ab1}, \cite{ab2}), while others are presented with more emphasis on applications (\cite{rcarmona}, \cite{hpam}, \cite{xyz}, \cite{wzhao }). Recently, the control theory found its new application in machine learning \cite{DL_control, Bao_DL}, which reveals even broader application scenarios for optimal control.

In most practically applications, finding an optimal control with closed form is difficult -- expect for some limited cases such like the linear-quadratic control problem, and people often need to obtain numerical solutions.
There are typically two types of numerical methods to solve the stochastic optimal control problem:  the dynamic programming and the stochastic maximum principle (SMP) \cite{p3}.
The dynamic programming approach aims to transfer the control problem to numerical solutions for a class of nonlinear PDEs (Hamilton-Jacobi-Bellman, i.e. HJB equations), whose solutions can be interpreted in the viscosity sense (\cite{wfms}, \cite{ntouzi}, \cite{hpam}). There are several successful numerical schemes designed to solve the HJB equation numerically (\cite{barles1},\cite{barles2}, \cite{jiewei}, \cite{weie}). The SMP approach, on the other hand, introduces an optimality condition for the optimal control. Then, the optimal control can be determined by an optimization procedure.  In order to solve the stochastic optimal control problem through SMP, a system of backward stochastic differential equations (BSDEs) is derived as the adjoint process of the controlled state, and hence obtaining numerical solutions for BSDEs is required.

In this paper, we focus on the SMP approach due to its advantages over the dynamic programming approach in two-folds: (i) SMP allows to have some state constraints; (ii) SMP allows to have random coefficients in the state equation and in the cost functional.
The general computational framework that we adopt is the gradient projection method \cite{main}, in which numerical schemes for BSDEs are used to calculate the gradient process of the cost functional, and the gradient projection optimization is applied to determine the optimal control (\cite{p3}, \cite{rcarmona}).

A major challenging in the existing gradient projection method for solving the stochastic optimal control problem lies in the spatial dimension approximation, which occurs in approximating conditional expectations for solutions of BSDEs. In low dimensional state spaces, one may compute  values of conditional expectation at tensor-product grid points, and then use polynomial interpolation to approximate the entire conditional expectation. However, it is well-known that tensor-product grid points with polynomial interpolation suffer from the curse of dimensionality, hence the computational cost increases exponentially as the dimension increases. To address the curse of dimensionality, and to improve the efficiency of the current gradient projection method, in this work we introduce meshfree approximation methods to implement spatial dimension approximation. 

The meshfree approximation methods typically avoid structured mesh grid points, which are required for polynomial interpolation. Therefore, meshfree methods are more flexible in embedding function features in the approximation for solving differential equations. Instead of using polynomial basis, meshfree approximation usually choose radial basis functions (RBFs) to calculate approximations \cite{fass}. In this paper, we first choose the moving least square method as our RBF method due to its efficiency and flexibility \cite{wen}. Then, we introduce general RBFs approximation and apply it to implement spatial dimension approximation in our computational framework. As a theoretical validation for our method, we shall carry out rigorous numerical analysis to prove the convergence of our meshfree method. Specifically, we shall systematically study the errors of meshfree approximation. Then, we incorporate spatial approximation errors caused by meshfree approximation into classical numerical analysis for stochastic optimal control and derive corresponding convergence results. To validate numerical effectiveness and efficiency, we shall carry out several numerical experiments to demonstrate the theoretical findings. 

The rest of this paper is organized as following. In Section \ref{GPM}, we give a brief introduction to the gradient projection method. In Section \ref{Algorithm}, we introduce a numerical algorithm that implements our meshfree approximation method to solve stochastic optimal control problems. Numerical analysis for our meshfree approximation method will be discussed in Section \ref{Analysis}, and we shall present numerical results in Section \ref{Numerics}.

\section{The gradient projection method for optimal control}\label{GPM}
\pindt Let $(\Omega, \mathcal{F}, \lbrace  \mathcal{F} \rbrace_{0 \leq t \leq T}, \mathbb{P} )$ be a filtered probability space of the multidimensional brownian motion $\boldsymbol{W}=(W_t)_{0 \leq t \leq T}$ in $\R^m$. Let ${\bf C}$ be a nonempty convex closed subset in $\R^{d_1}$, and we define the control set $U$ as following
\[
U = \lbrace u \in L^2([0,T];\R^{d_1}) \ | \ u(t) \in {\bf C} \ a.e. \rbrace.
\]
Consider the following dynamics of the state variable $x_t \in \R^d$, which is controlled by some process $u_t \in U$, 

\label{sde}
\begin{equation}
	dx_t^u=b(\xtu, u_t) dt  + \sigma(\xtu,u_t)dW_t.
\end{equation}
In this work, we assume that the control process is deterministic, 
and the coefficients $ (b,\sigma): \R^d \times U \rightarrow \R^d \times \R^{d \times m}$
are assumed to be \textit{continuously differentiable with bounded derivatives.}

The cost functional is defined as follows:
\begin{equation}
	J(u)=\E[\int_0^T j(x^u_t,u_t)dt + k(x^u_T)]
\end{equation}
where $j$ is the running cost $j : \R^d \times \R^{d_1} \rightarrow \R$ and  $k$ is the terminal cost $k : \R^d \rightarrow \R$. 
We assume that $j$ is continuously differentiable and the derivatives have at most linear growth in the underlying variables. 

The stochastic optimal control problem that we want to solve in this work is defined by
\begin{equation}
	\text{Find } u^* \in U  \text{ such that }  J(u^*) = \min_{u \in U} J(u).
\end{equation}
To find the optimal control $u^*$, we introduce a gradient projection method, which is based on the following formula
\begin{equation}{\label{cont_projection}}
	u_t^*=P_U(u_t^* - \rho J'(u^*)|_t), 
\end{equation}
where $P_U$ is the projection on to the subspace $U \subset L^2([0,T];\R^{d_1})$. 

Introducing the process $(\bp_t, \bq_t)$, the gradient $J'(u)$ can be derived through taking directional derivative (\cite{p2}, \cite{p3}, \cite{rcarmona}, \cite{xyz}),
\begin{align}
	J'(u) = \lim_{\epsilon\rightarrow 0 } \frac{J(u+\epsilon(\tilde{u}-u))-J(u) }{\epsilon}
\end{align}
  and the expression of $J'(u)$ can be found to be  
\begin{equation} \label{cont_grad}
	J'(u)|_t=\E[(\bp_t)^\dagger b'_u+ tr((\bq_t)^\dagger \sigma'_u)+j'_u(u(t))]
\end{equation}
where $(\cdot)^{\dagger}$ stands for the transpose operator, and the stochastic processes $ (x^u_t, \bp_t^u, \bq_t^u)$ in \eqref{cont_grad} are solutions of the following forward backward stochastic differential equations (FBSDEs) system
\begin{align} 
dx_t&=b(x_t,u)dt +\sigma(x_t,u)dW_t , \qquad &x_{t=0}=x_0 {\label{SDE}}\\
-d\bp_t&=\big(D_x j(x,u) + D_xb(x,u)  \bp_t + tr(\bq_t^{\dagger}D_x\sigma(x,u)\big) dt - \bq_t dW_t , \ & \bp_T=D_x k(x_T) {\label{BSDE}}
\end{align}
where $D_x$ stands for partial differentiation with respect to the space variables, and we use bold face letters $\bp$ and $\bq$ to emphasis high dimensionality of the solutions of the BSDE \eqref{BSDE}. 
For notational convenience, we denote
\begin{equation*}
	f(x,\bp,\bq,u):=D_x j(x,u) + D_xb(x,u)  \bp + tr(\bq^{\dagger}D_x\sigma(x,u)).
\end{equation*}

As a result, finding the optimal control hinges on solving \eqref{cont_projection}, which can be achieved by using an iterative scheme.
In the next section we will design a numerical algorithm to find the optimal control.

\section{Numerical algorithm for the projection method}\label{Algorithm}
To proceed, we first define a discretized control space, and then introduce the discretization for the system of FBSDEs. We will also develop an efficient meshfree approximation based spatial discretization scheme, which is a challenging task when solving the optimal control problem. 
\subsection{Temporal discretization for the optimal control}
For a positive integer $N$, we introduce the following temporal discretization over $[0, T]$
$$
0=t_0 < t_1 < \cdot \cdot \cdot < t_N=T, \ \ t_{n+1}-t_{n}=\Dt =\frac{T}{N}, \ \ I_n^N:=[t_{n-1},t_{n}).
$$
For control processes, we define a discretized control space as a subspace of $U$, i.e.
$$
U_N= \Big \{ u \in U \ | \ u=\sum_{m=1}^N \alpha_m  \chi_{I^N_m} ,\ a.e  ,\ \alpha_m \in \R^{d_1} \Big \},
$$
and the optimal control problem on $U_N$ becomes
$$
J(u^{*,N}) = \min_{u \in U_N} J(u).
$$
Then, the projection formula becomes
\begin{equation}{\label{dis_projection1}}
	u^{*,N}=P_{U_N}(u^{*,N} - \rho J'(u^{*,N})). 
\end{equation}
For pre-chosen optimization step-sizes $\{\rho_i\}_i$ and an initial guess $u^{0,N}$, we introduce the following fixed point iteration scheme to determine the optimal control $u^{*, N}$
\begin{equation}\label{fixed_point}
	u^{i+1,N}=P_{U_N}(u^{i,N} - \rho_i J_N'(u^{i,N})),
\end{equation}
and we denote the error between the operator $J'$ and $J_N'$ by 
\begin{equation}
	\epsilon_N= \sup_i ||J'(u^{i,N}) - J_N'(u^{i,N})||.
\end{equation}
The following theorem guarantees that $u^{*,N}$ converges to the true optimal control $u^{*}$ under some assumptions. 
 \begin{theorem} {\label{thm1}}
 	Assume that $J'(\cdot)$ is lipschitz and uniformly monotone around $u^*$ and $u^{*,N}$ in the sense that there exist positive constants $c$ and $C$ such that 
 	\begin{align*}
 		& ||J'(u^*)- J'(v)|| \leq C ||u^* -v||, \ \ \   \forall v \in U. \\
 		 & (J'(u^*)- J'(v), u^*-v) \geq c||u^*-v^*||^2, \ \ \ \forall v \in U. \\
 		 &||J'(u^{*,N})- J'(v)|| \leq C ||u^{*,N} -v||, \ \ \ \forall v \in U_N .\\
 		& (J'(u^{*,N})- J'(v), u^{*,N}-v) \geq c ||u^{*,N}-v^*||^2, \ \ \ \forall v \in U .
 	\end{align*}
 Also, assume that the operator $J'_N$ is unbiased, i.e. $\epsilon_N  \rightarrow 0,  \   N\rightarrow \infty$, and
$\rho_i$ is picked such that $0 < 1 - 2 c \rho_i +(1+2C) \rho_i^2 \leq \sigma^2$ for some $0 < \sigma < 1$. Then the iteration scheme \eqref{fixed_point} is convergent, that is
\[
||u^{*}-u^{i,N}|| \rightarrow 0, \ i, N \rightarrow \infty.
\]
 \end{theorem}
 \begin{remark}
 	 We refer to \cite{main} for the proof of Theorem \ref{thm1}, in which one will also find that 
 	 \begin{equation}
 	 	||u^{*}-u^{i,N}|| \sim \oo(\Delta t),  \ \  i,N \rightarrow \infty, \ \ \text{if} \ \ \epsilon_N \sim \oo(\Delta t).
 	 \end{equation}
 \end{remark}

\subsection{Numerical schemes for FBSDEs}
In order to compute the gradient $J'_N$ in the iteration scheme \eqref{fixed_point}, we need to solve the BSDE \eqref{BSDE} numerically. To this end, we integrate the BSDE over the time interval $[t_n, t_{n+1}]$ and obtain the following integral form
\begin{equation}{\label{intbsde}}
	\bp_{t_n}=\bp_{t_{n+1}} + \int_{t_n}^{t_{n+1}} f(x_t,\bp_t,\bq_t,u_t) dt - \int_{t_n}^{t_{n+1}} \bq_t dW_t.
\end{equation}
Taking conditional expectation $\Exn[\cdot]:=\E[\cdot | x_{t_n}=x]$ on both sides of the above equation and apply the left-point formula to approximate the deterministic integral, we obtain
\begin{equation}\label{exactp}
	\bp^x_{t_n}=\Exn[\bp_{t_{n+1}}] + \Dt f(x, \bp^x_{t_n},\bq^x_{t_n},u_{t_n}) + \bar{\bR}^x_{p,n} ,
\end{equation}
where $\bar{\bR}^x_{p,n}:=\int_{t_n}^{t_{n+1}}\Exn[f(x_t,\bp_t,\bq_t,u_t)] dt - \Dt f(x, \bp^x_{t_n},\bq^x_{t_n},u_{t_n})$ is the truncation error, and the stochastic integral is eliminated due to the martingale property of It\^o integral.

To derive an approximation for $\bq$, we multiply both sides of (\ref{intbsde}) by $\Delta W_{t_{n+1}}:= W_{t_{n+1}} - W_{t_n}$ and get
$$ \bp_{t_n} \Delta W_{t_{n+1}} =\bp_{t_{n+1}} \Delta W_{t_{n+1}}+ \int_{t_n}^{t_{n+1}} f\Delta W_{t_{n+1}} dt - \int_{t_n}^{t_{n+1}} \bq_t dW_t \Delta W_{t_{n+1}}. $$
Taking conditional expectation on both sides of the above equation, we have
\begin{equation}{\label{exactq}}
	\bq^x_{t_{n+1}}=\frac{1}{\Dt} \big( \Exn [\bp_{t_{n+1}}\Delta W_{t_{n+1}}] + \bar{\bR}^x_{q,n} \big),
\end{equation}
where
\begin{equation}
	\bar{\bR}^x_{q,n}=\int_{t_n}^{t_{n+1}} \Exn[f\Delta W_{t_{n+1}} ]dt-\int_{t_n}^{t_{n+1}} \Exn[\bq_t]dt + \Dt \bq^x_{t_{n}}.
\end{equation}
In conclusion, we introduce the following temporal discretization scheme for a given spatial point $x \in \mathbb{R}^d$
\begin{align}
	\bp^x_n&=\Exn[\bp_{n+1}] + \Dt f(x,\bp^x_n,\bq^x_n,u_{t_n}), \nonumber \\
	\bq^x_n&=\frac{1}{\Dt} \Exn[\bp_{n+1}  \Delta W_{n+1}]. \label{semi_dis}
\end{align}
That is, for each pair $(t_n, x)$, solutions $\bp_{t_n}^x, \bq_{t_n}^x$ are approximated by $\bp^x_n, \bq^x_n$. 

\subsection{Approximation of $\Exn[\cdot]$}

In the system of equations \eqref{semi_dis}, we can see that the computation of $\bp^x_n, \bq^x_n$ requires computing the conditional expectation $\Exn[\cdot]$, and it also requires approximations of $\bp_{n+1}, \bq_{n+1}$ over the state space for $x_t$.  Typically, people compute values of those functions at a collection of spatial points
\begin{equation*}{\label{xpoints}}
	 X := \lbrace x_k \rbrace^M_{k=1} \subset \R^d,
\end{equation*}
where $M$ is the total number of the spatial points, and then use interpolatory values to approximate the entire function. 

In one dimensional problem, the spatial interpolation is realized by using simple polynomial interpolation. 
However, classical polynomial interpolation methods, together with the uniform tensor-product grid points, suffer from the curse of dimensionality and may turn out to be a rather inefficient method in even moderate dimensions such like $d=4$ or $5$. 

In this paper, we use meshfree points as our choice of spatial points, on which we solve for the FBSDEs system \eqref{SDE}-\eqref{BSDE}, and we shall apply meshfree approximation as our interpolation method to construct the interpolatory approximation for the entire function -- due to its high flexibility/high efficiency advantages compared with the standard tensor grids based polynomial interpolation methods.
The collection of meshfree points $X$ will be $d$-dimensional Halton sequence points in this work, 
and we explore two methods for spatial interpolation: the moving least square method (MLS) and the radial basis function (RBF) interpolation.

To proceed, we denote $I_h \phi(x)$ as our interpolatory approximation operator for a function $\phi(x)$ at the point $x \in \R^d$, i.e.
\begin{equation}
	\phi(x) \approx I_h \phi(x) :=\sum_{i \in \mathcal{I}(x)} a^*_i(x) \phi(x_i)
\end{equation}
with $ a^*_i(x)$ defined pointwise in terms of $x$. 

And we consider two spatial approximation schemes for $I_h$
\begin{enumerate}
	\item The moving least square method (MLS), and it is defined in \eqref{mls_lagrange}. 
	\item The radial basis function interpolation method (RBF), and it is defined in \eqref{rbf_f}. 
\end{enumerate}

Here $\mathcal{I}(x)$ is the collection of indices of spatial points that we use to approximate $\phi(x)$. 
One important concept in scattered data approximation is the fill distance.
\begin{align}{\label{fill_distance}}
	h_{X,\Omega} :=  \sup_{x \in \Omega} \min_{1 \leq k \leq M} ||x-x_k||_2
\end{align}

 Such distance measures the ``density" of the data, and it can be interpreted as the radius of the largest ball that can be placed in the domain without intersecting the data points \cite{fass}. In later sections, we will use $h$ or $h_{X,\Omega}$ interchangeably to denote the fill distance, and we will show that both MLS and RBF approximations will give us the desired level of interpolation accuracy, i.e. $\oo(h^2)$ in this work.

In what follows, we shall derive our approximation method to compute conditional expectations. 
First of all, we simulate the dynamics of $x_{t_{n+1}}$ by using the Euler - Maruyama scheme, and we let
\begin{equation}{\label{cdep1}}
	\tilde{x}^{x,t_n}_{t_{n+1}} = x + b(x,u(t_n)) \Dt + \sigma(x, u(t_n)) \Delta W_{n+1}.
\end{equation}
Since $\Delta W_{n+1} \sim \sqrt{\Dt} \zeta$, $\zeta \sim \mathcal{N}(0,I)$, the above simulation can be written as
\begin{equation}
	\tilde{x}^{x, t_n}_{t_{n+1}} =x +b(x,u(t_n)) \Dt + \sigma(x, u(t_n))\sqrt{\Dt} \zeta.
\end{equation}
As a result, the conditional expectation can be written as a $d$ dimensional integral
\begin{equation}{\label{cdep}}
	\Exn[\tilde{\bp}_{t_{n+1}}] := \Exn[\bp_{t_{n+1}}(\tilde{x}^{x, t_n}_{t_{n+1}})]= \int \bp_{t_{n+1}}\Big( x +b(x,u(t_n)) \Dt + \sigma(x, u(t_n)) \sqrt{\Dt} \xi \Big)  \rho(\xi)d \xi,
\end{equation} 
where $\rho(\xi) \sim e^{-\frac{1}{2} \xi^T \Sigma^{-1} \xi}$ is the probability density function of the multivariate Gaussian distribution. In this approach, we use Guass-Hermite quadrature to approximate the integral ($\ref{cdep}$) due to its high efficiency in approximating moderate dimensional integrals, and we get
\begin{align*}
	\Exn[\tilde{\bp}_{t_{n+1}}] \approx \tilde{\E}^x_{t_n}[\bp_{t_{n+1}}] := \sum^L_{l_1=1} ... \sum^L_{l_d=1}  \bp_{t_{n+1}} \Big(x +b(x,u(t_n)) \Dt +  (\sum _{j} \sigma_{\cdot,j}(x, u(t_n)) \sqrt{\Dt} \xi^j_{l_j}  )\Big) \prod_{j=1}^d\omega^j_{l_j}  ,
\end{align*}
where $\{\xi^{j}_{l_j}\}$ and $\{ \omega^j_{l_j}\}$ are Guass-Hermite points and weights. 
Since the simulated value $\tilde{x}^{x, t_n}_{t_{n+1}}$
may not coincide with any of the specified spatial points,  
we will use the interpolated value $I_h \bp_{t_{n+1}}(\tilde{x}^{x, t_n}_{t_{n+1}})$ to approximate $\bp_{t_{n+1}}(\tilde{x}^{x, t_n}_{t_{n+1}})$, hence we define
\begin{equation*}
	\hat{\E}_{t_n}^x[\bp_{t_{n+1}}]:=\sum^L_{l_1,\cdots, l_d=1}  I_{h}\bp_{t_{n+1}}\Big(x +b(x,u(t_n)) \Dt + \big [\sum _{j} \sigma_{\cdot,j}(x, u(t_n)) \sqrt{\Dt} \xi^j_{l_j} \big ] \Big) \prod_{j=1}^d\omega^j_{l_j}
\end{equation*}
as our approximation for $\E_{t_n}^x[\bp_{t_{n+1}}]$ with interpolatory approximation for $\bp_{t_{n+1}}$. In this paper, we use MLS and RBF as meshfree interpolation methods to calculate $I_h$ and we shall discuss their approximation errors in the next section. We want to mention that other approximation methods, such like Monte Carlo method, can also be used to approximate ($\ref{cdep}$) when the dimension of the problem is very high. However, the difficulty of interpolation remains for most approaches and the main effort of our research in this work is to address the challenge of high dimensional function approximation.

In the scheme for $\bq_{t_n}^x$, we denote by $\tilde{\E}^x_{t_n}[\bp_{t_{n+1}} \Delta W^T_{n+1}]$ the approximation for $\E^x_{t_n}[\tilde{\bp}_{t_{n+1}} \Delta W^T_{n+1}]$ and we use $\tilde{\bp}^i_{t_{n+1}} \Delta W^{k,T}_{n+1}$ to denote the $(i,k)$-th component in the matrix. Therefore, 
\begin{equation*}
	\tilde{\E}^x_{t_n}[\bp^i_{t_{n+1}} \Delta W^{k,T}_{n+1}] := \sum^L_{l_d=1}... \sum^L_{l_1=1} \bp^i_{t_{n+1}}\Big(x +b(x,u(t_n)) \Dt + \big[\sum _{j} \sigma_{\cdot,j}(x, u(t_n)) \sqrt{\Dt} \xi^j_{l_j} \big] \Big) \sqrt{\Dt} \zeta^k_{l_k}\prod_{j=1}^d\omega^j_{l_j},
\end{equation*}
and the interpolatory approximation for the above expectation is 
\begin{equation*}
\hat{\E}^x_{t_n}[\bp^i_{t_{n+1}} \Delta W^{k,T}_{n+1}] :=\sum^L_{l_1,\cdots, l_d=1} I_h \bp^i_{t_{n+1}} \Big(x +b(x,u(t_n)) \Dt + \big [\sum _{j} \sigma_{\cdot,j}(x, u(t_n)) \sqrt{\Dt} \xi^j_{l_j} \big ] \Big) \sqrt{\Dt} \zeta^k_{l_k}\prod_{j=1}^d\omega^j_{l_j}.
\end{equation*}

As a result, we have the following approximations 
\begin{align} \label{condit}
	\E_{t_n}^x[\bp_{t_{n+1}}]&=\hat{\E}^x_{t_n}[\bp_{t_{n+1}}]+\hat{\bR}^x_{p,n}, \nonumber \\
	\E_{t_n}^x[\bp_{t_{n+1}} \Delta W^T_{n+1}]&=\hat{\E}^x_{t_n}[\bp_{t_{n+1}}\Delta W^T_{n+1}]+\hat{\bR}^x_{q,n}, 
\end{align}
where
\begin{align*}
	\hat{\bR}^x_{p,n} :=\tilde{\bR}^x_{p,n}+\bR^x_{E,p,n}+\bR^x_{I,p,n}, \\
	\hat{\bR}^x_{q,n} :=\tilde{\bR}^x_{q,n}+\bR^x_{E,q,n}+\bR^x_{I,q,n} 
\end{align*}
are approximation errors, which are composed of three parts: the Euler approximation error for the state $x_t$, i.e. $\tilde{\bR}^x_{\cdot,n}$, the Gauss-Hermite quadrature error, i.e. $\bR^x_{E,\cdot,n}$, and the function approximation error from spatial interpolation, i.e. $\bR^x_{I,\cdot,n}$. Specifically, the error terms in $\hat{\bR}^x_{p,n}$ are defined as follows
	$$\tilde{\bR}^x_{p,n}= \E^{x}_{t_{n}} [\bp_{t_{n+1}}]-\E^{x}_{t_{n}}[\tilde{\bp}_{t_{n+1}}], \ \bR^x_{E,p,n}=\E^{x}_{t_{n}}[\tilde{\bp}_{t_{n+1}}]- \tilde{\E}^{x}_{t_n}[\bp_{t_{n+1}}], \ \bR^x_{I,p,n}= \tilde{\E}_{t_n}^x[\bp_{t_{n+1}}]-\hat{\E}_{t_n}^x[\bp_{t_{n+1}}].$$
The error terms in $\hat{\bR}^x_{q,n}$ are defined as follows
	$$\tilde{\bR}^x_{q,n}= \E^{x}_{t_{n}} [\bp_{t_{n+1}} \Delta W^T_{n+1}]-\E^{x}_{t_{n}}[\tilde{\bp}_{t_{n+1}} \Delta W^T_{n+1}], \ \bR^x_{E,q,n}=\E^{x}_{t_{n}}[\tilde{\bp}_{t_{n+1}}\Delta W^T_{n+1}]- \tilde{\E}^{x}_{t_n}[\bp_{t_{n+1}}\Delta W^T_{n+1}],$$
	$$ \ \bR^x_{I,q,n}= \tilde{\E}_{t_n}^x[\bp_{t_{n+1}}\Delta W^T_{n+1}]-\hat{\E}_{t_n}^x[\bp_{t_{n+1}}\Delta W^T_{n+1}].$$ 	
%

\subsection{Fully discretized schemes}
Based on the above discussions, we rewrite the approximation schemes \eqref{exactp}, \eqref{exactq} as following 
\begin{equation}\label{pde_exact} 
\begin{aligned}
	\bp^x_{t_n}&=\hat{\E}^{x}_{t_{n+1}}[\bp_{t_{n+1}}] + \Dt f(x,\bp^x_{t_n},\bq^x_{t_n},u_{t_n}) +\bR^x_{p,n} , \ \bp^x_{t_N}=g(x)  \\
	\bq^x_{t_n}&=\frac{1}{\Dt} (\hat{\E}^{x}_{t_{n+1}}[\bp_{t_{n+1}} \Delta W^T_{n+1}]+\bR^x_{q,n})
\end{aligned}	
\end{equation}
where $\bR^x_{p,n}=\bar{\bR}^x_{p,n}+\hat{\bR}^x_{p,n}$ , $\bR^x_{q,n}=\bar{\bR}^x_{q,n}+\hat{\bR}^x_{q,n}$. 

By dropping the error terms, we propose the following fully discretized scheme for solving the system of BSDE on a selection of spatial points $\{x_k\}_{k=1}^{M}$
\begin{equation}{\label{pde_num}}
\begin{aligned}
	\bp^{x_k}_n&=\hEkn[\bp_{n+1}] + \Dt f({x_k},\bp^{x_k}_n,\bq^{x_k}_n,u_{t_n}), \\
	\bq^{x_k}_n&=\frac{1}{\Dt} (\hEkn  [\bp_{n+1} \Delta W^T_{n+1}]),
\end{aligned}	
\end{equation}
where $\bp_n, \bq_n$ are interpolatory approximations described as 
\begin{equation}{\label{pwpq}}
\begin{aligned}
	\bp^i_{n}(x) & = \sum_{k \in \mathcal{I}_h(x)} a^i_{\bp_n, k}(x) \bp_{n}^{i,x_k}, \\
	\bq^{ij}_n(x) & = \sum_{k \in \mathcal{I}_h(x)} a^{ij}_{\bq_n,k}(x) \bq_{n}^{ij,x_k}.
\end{aligned}	
\end{equation}
Here $a^i_{\bp_n, k}(x),a^{ij}_{\bq_n, k}(x)$ are the coefficients obtained from interpolating the function $\bp^i_n,\bq^{ij}_n$ at spatial location $x^k$, and the upper indices indicate the $i$th (resp. $ij$th) component of $\bp$ (resp. $\bq$). Also, $\mathcal{I}_{h}(x)$ is the collection of the indices of neighborhood points of $x$.

 When the radial basis interpolation method is used, $\mathcal{I}_h(x)=\lbrace k \rbrace^M_{k=1}$ since all the data points will be used. However, when we are using the moving least square interpolation method, as what will be stated in Theorem \ref{lagrange_form}, it consists of the neighborhood points that  reproduce the local polynomial functions of order $l$: 
 \begin{align} {\label{mls_reduce}}
 B^k &= \lbrace x_j \big | j \in \lbrace 1, ..., M\rbrace , \ ||x-x_j||_2 < r \rbrace \\
		\mathcal{I}_h(x) &= \big\lbrace i \big | \lbrace x_i \rbrace^I_{i} \in B^k  \text{ is } \pi_l(\R^d) \text{-unisolvent} \big \rbrace
	\end{align}
	$r$ is the radius of the compact support for the weight function as in \eqref{mls_initial_def}.



\subsection{Summary of the numerical algorithm}
Now, we give a complete description of the numerical schemes for implementing optimization scheme \eqref{fixed_point}. 
Since our goal is to find the gradient $J'_N$, which is under expectation, it is natural to design an approximation operator for $\E$. We define the operator $\hE$ to be the following:
\begin{equation}{\label{ehat}}
	\hE[\phi_{t_0}]=\phi_{t_0}, \  \hE[\phi_{t_n}]=\hE^{x_0}_{t_0}[\hE_{t_1}[... \hE_{t_{n-1}}[\phi_{t_n}]]],
\end{equation}
and we use Monte Carlo sampling to evaluate the above expectation as introduced in \cite{main}, and the function values are obtained by using our meshfree interpolation methods. The spatial locations are typically chosen to be Halton sequences in the corresponding dimension which are used in conjunction with radial base interpolation functions when the dimension is high. And it is because such approximation strategy are demonstrated to be more efficient than the tensor-grid approximation methods when the dimension is high. \cite{wen}, \cite{fass}. 
Therefore, the gradient operator $J'(u)|_{t_n}$ is approximated by
\begin{equation}{\label{project1}}
	J_N'(u)|_{t_n}= \hE[\bp_n b'_u(\cdot, u_{t_n})) + tr\big( \bq_n^{\dagger} D_u\sigma(\cdot,u_{t_n})\big)] + j'(u_{t_n}).
\end{equation}


In what follows, we summarize our gradient projection method.

\noindent \textbf{The Gradient Projection Algorithm using MLS}
\vspace{0.3em}

\noindent Choose a tolerance $\epsilon_0$, an initial guess $u_0 \in U_N$, the meshfree spatial point set $X=\lbrace x_k \rbrace$ and a weight function $\omega$ as in \eqref{mls_initial_def}. 
\begin{enumerate}
	\item Set the terminal condition $\bp_N^k=g(x_k)=k'(x_T)$. 
	\item Compute $(\bp_n, \bq_n)$ for each $n=N-1, ... , 0$ by the schemes \eqref{pde_num}:

Compute $(\bp^{x_k}_n, \bq^{x_k}_n)$ at each space location and define $(\bp,\bq)$ through \eqref{pwpq}, where the coefficients are found by using the form in Theorem \ref{lagrange_form}, where the numerical implementation is given in \cite{fass} chapter $23$. 
	\item Approximate the gradient $J_N'(u)|_{t_n}$ by (\ref{project1}). 
	\item Carry out the optimization procedure \eqref{fixed_point} with approximated gradient in Step 3.
%
	\item Go back to step 1 until the tolerance is reached. 
\end{enumerate}

 \textbf{The Gradient Projection Algorithm using RBF}
\vspace{0.3em}

\noindent Choose a tolerance $\epsilon_0$, an initial guess $u_0 \in U_N$ and the meshfree spatial point set $X=\lbrace x_k \rbrace$ and a radial basis $\phi$. 
\begin{enumerate}
	\item Set the terminal condition $\bp_N^k=g(x_k)=k'(x_T)$. 
	\item Compute $(\bp_n, \bq_n)$ for each $n=N-1, ... , 0$ by the schemes \eqref{pde_num}:

Compute $(\bp^{x_k}_n, \bq^{x_k}_n)$ at each space location and define $(\bp,\bq)$ through \eqref{pwpq}, where the coefficients are found by using \eqref{psetup}. 
	\item Approximate the gradient $J_N'(u)|_{t_n}$ by (\ref{project1}). 
	\item Carry out the optimization procedure \eqref{fixed_point} with approximated gradient in Step 3.
%
	\item Go back to step 1 until the tolerance is reached. 
\end{enumerate}

\section{Error Analysis for the algorithm}\label{Analysis}
The general computational framework of our meshfree approximation method follows the gradient projection approach for solving the stochastic optimal control problem. The primary contribution of this work is the application of meshfree approximation to approximate high dimensional functions, therefore improve the efficiency of the gradient project approach in solving higher dimensional problems. As a theoretical validation for our effort, in this section we give a detailed error analysis for our meshfree approximation algorithm. 

Since the approximation error for the gradient $J'_N$ is composed of spatial approximation errors and temporal approximation errors, we will first give some properties of the operator  $\hat{\E}$,  which are directly related to the spatial approximation errors. Then, we will analyze the (spatial-temporal) errors in approximating solutions $(\bp,\bq )$ of the BSDE. Finally, we will combine both spatial analysis and temporal analysis to obtain the approximation errors for $J'_N(u^{i,N})$, and therefore derive the error analysis for the optimal control.

\subsection{Estimation for meshfree approximation}
In this subsection, we focus on the analysis for meshfree approximation. The following lemma gives some basic properties satisfied by the meshfree approximation operator $\hE$, which will be used in the numerical analysis for the optimal control problem. These properties are analogues of standard properties for expectation. 

To prove Lemma \ref{operator_error}, we make some assumptions on the basis of interpolation functions.
 
Both spatial the interpolation methods are known to take Lagrangian forms : 
for $f \in C(\R^d)$
\begin{equation}
	I_h f(x) = \sum^M_{k=1} f(x_k) \Phi_k(x_k, x)
\end{equation}
where for MLS, $\lbrace \Phi_k(x_k, x) \rbrace^M_{k=1}$ are defined to be \eqref{mls_lagrange} in Theorem \ref{lagrange_form} (Also see \cite{wen}, \cite{fass}); and for RBF, it is defined in Equation (11.1) in \cite{wen} Chapter 11, which takes a different form than \eqref{rbf_f}. Also, notice that by definition the basis $\lbrace \Phi_k(x_k, x) \rbrace^M_{k=1}$ don't depend on $\lbrace f_k(x_k) \rbrace^M_{k=1}$.

We make the following assumption on $\lbrace \Phi_k(x_k, x) \rbrace^M_{k=1}$: 
\begin{enumerate}
	\item $\Phi_k(x_k, x)  \geq 0, \ \forall  \ k \in \lbrace 1, ..., M\rbrace$,
	\item $\sum^M_{k=1} \Phi_k(x_k, x) \leq 1 $.
\end{enumerate}
We point out that such assumptions are satisfied by some widely used scattered data approximation methods, e.g. the Shepard's method, which takes the following form: 
\begin{equation}
	I_h f(x):= \sum^M_{k=1} f(x_k) \frac{\omega(x_k,x)}{\sum^M_{k=1} \omega(x_k,x)},
\end{equation}
where the function $\omega(x,y)$ is a positive weight function typically with compact support.

\begin{lemma}{\label{assumptions}}
	Let $\phi_{t_{n+1}}=\phi(t_{n+1},x_{t_{n+1}})$, we have 
	\begin{enumerate}
		\item $\hE [\hE_{t_n}[\phi_{t_{n+1}}]]=\hE[\phi_{t_{n+1}}]$,
		\item $(\hEkn[\phi_{t_{n+1}}])^2 \leq \hEkn[(\phi_{t_{n+1}})^2]$,
		\item $(\hEkn[\phi_{t_{n+1}} \Delta W_{t_{n+1}}])^2 \leq \big( \hEkn[(\phi_{t_{n+1}})^2 ]-(\hEkn[\phi_{t_{n+1}}])^2 \big) \Delta t$,
		\item (Monotonicity) If $\phi_{t_{n+1}}^x \geq 0 $  for any $x$, then $\hEkn[\phi_{t_{n+1}}] \geq 0$, and so $\hE[\phi_{t_{n+1}}] \geq 0$.
	\end{enumerate}
\end{lemma}

\begin{proof}
	For notational simplicity, we prove only the case when $d=1$, and the proof for multi-dimensional case is analogous. 

	Part 1 of the Lemma is true due to \eqref{ehat}. 
	
	For Part 2, notice that
	\begin{align}
		(I_h \phi_{t_{n+1}}(x))^2 &= \Big( \sum^M_{k=1} f(x_k) \Phi_k(x_k,x) \Big)^2 \nonumber \\
		& \leq \Big( \sum^M_{k=1} f(x_k)^2 \Phi_k(x_k,x) \Big) (\sum^M_{k=1} \Phi_k(x_k,x)) \nonumber \\ 
		& \leq \sum^M_{k=1} f(x_k)^2 \Phi_k(x_k,x) \\
		& =I_h \phi^2_{t_{n+1}}(x), \label{lemma_last1}
	\end{align}
where we have used Cauchy's inequality and the assumptions made previously. 
	Now by definition, we have
	\begin{align}
		(\hEkn[\phi_{t_{n+1}}])^2 &\leq  \Big(\sum^L_{i=1} I_h \phi_{t_{n+1}}(x_k+b(x_k,u_{t_n})\Delta t +\sigma(x_k,u_{t_n})\sqrt{\Delta t} \xi_i ) \omega_i\Big)^2 \nonumber \\ 
		& \leq \sum^L_{i=1} \Big( I_h \phi_{t_{n+1}}(x_k+b(x_k,u_{t_n})\Delta t +\sigma(x_k,u_{t_n})\sqrt{\Delta t} \xi_i )\Big)^2 \omega_i \sum^L_{i=1} \omega_i \nonumber \\
		& \leq \sum^L_{i=1}  I_h \Big(\phi_{t_{n+1}}(x_k+b(x_k,u_{t_n})\Delta t +\sigma(x_k,u_{t_n})\sqrt{\Delta t} \xi_i )\Big)^2 \omega_i, \label{lemma_last}
	\end{align}
where in \eqref{lemma_last}, we used \eqref{lemma_last1}. And the conclusion holds. 

For part 3 of the lemma, 
 \begin{align}
 	(\hEkn[\phi_{t_{n+1}} \Delta W_{t_{n+1}}])^2 & = \Big( \sum^L_{i=1} I_h \phi_{t_{n+1}}(x_k+b(x_k,u_{t_n})\Delta t +\sigma(x_k,u_{t_n})\sqrt{\Delta t} \xi_i)\sqrt{\Delta}\xi_i \omega_i\Big)^2 \nonumber \\ 
 	& \leq \Delta t \Bigg( \sum^L_{i=1} \Big( I_h \phi_{t_{n+1}}(x_k+b(x_k,u_{t_n})\Delta t +\sigma(x_k,u_{t_n})\sqrt{\Delta t} \xi_i )\Big)^2 \omega_i \nonumber \\
 	&- \Big(\sum^L_{i=1}  I_h \phi_{t_{n+1}}(x_k+b(x_k,u_{t_n})\Delta t +\sigma(x_k,u_{t_n})\sqrt{\Delta t} \xi_i ) \omega_i \Big)^2 \Bigg) \label{lemma_first_stage} \\
 	& \leq \Delta t \Bigg( \sum^L_{i=1}  I_h \Big(\phi_{t_{n+1}}(x_k+b(x_k,u_{t_n})\Delta t +\sigma(x_k,u_{t_n})\sqrt{\Delta t} \xi_i )\Big)^2 \omega_i \nonumber \\ 
 	& +\Big(\sum^L_{i=1}  I_h \phi_{t_{n+1}}(x_k+b(x_k,u_{t_n})\Delta t +\sigma(x_k,u_{t_n})\sqrt{\Delta t} \xi_i ) \omega_i \Big)^2 \Bigg) \nonumber \\
 	& = \big( \hEkn[(\phi_{t_{n+1}})^2 ]-(\hEkn[\phi_{t_{n+1}}])^2 \big) \Delta t,
 \end{align}
where \eqref{lemma_first_stage} can be obtained by applying exactly the same argument as the Proofs in \cite{BCZ_2015, Bao_bdsde, BCZ_2018}.

For part 4 of the Lemma 
set: $\tilde{x}_k =x_k+b(x_k,u_{t_n})\Delta t +\sigma(x_k,u_{t_n})\sqrt{\Delta t} \xi_i $, then we have:
\begin{align}
	\hEkn[\phi_{t_{n+1}}] &= \sum^L_{i=1} I_h \phi_{t_{n+1}}(x_k+b(x_k,u_{t_n})\Delta t +\sigma(x_k,u_{t_n})\sqrt{\Delta t} \xi_i ) \omega_i \nonumber \\ 
	&=\sum^M_{k=1} \phi_{t_{n+1}}(\tilde{x}_k) \Phi(\tilde{x}_k, x_k ) \omega_i .
\end{align}
Then, by assumption since $\phi_{t_{n+1}}(\tilde{x}_k) \geq 0$	for any $\tilde{x}_k$, the assertion is proved. 
\end{proof}


The following result gives the approximation error of $\hE[\phi_{t_n}]$ in approximating the expectation $\E[\phi_{t_n}]$. 
\begin{lemma}{\label{operator_error}}
	Assume that $b , \sigma \in C^{0,4}_b$. For $\phi_t = \phi(t,x_t) \in C^{0,4}_b$, we define $\Phi_{t_i}(x)=\E^x_{ti}[\phi_{t_n}]$. Then it holds that $\Phi_{t_i} \in C^{0,4}_b$, and we have 
	\begin{equation}\label{Lem:exp}
	\E[\phi_{t_n}]=\hE[\phi_{t_n}] + \sum_{i=0}^{n-1} \hE[\hat{R}_{\Phi,i}],
	\end{equation}
	where $\hat{R}_{\Phi,i}=\E^{x_{t_i}}_{t_i}[\Phi_{t_{i+1}}]-\hE^{x_{t_i}}_{t_i}[\Phi_{t_{i+1}}]$, $1 \leq i \leq n$. 
\end{lemma}
\begin{proof}
	To prove the estimate \eqref{Lem:exp}, we use induction.
	
	First of all, we can see that \eqref{Lem:exp} holds for $n=1$. Suppose that \eqref{Lem:exp} holds for $n$, we need to check that the equality also works for $n+1$. 
	Define $\psi_{t_n}=\E_{t_n}^{x_{t_n}}[\phi_{t_{n+1}}]$, then we have
		\[
	\E[\psi_{t_n}]=\hE[\psi_{t_n}] + \sum_{i=0}^{n-1} \hE[\hat{R}_{\Phi,i}],
	\]
	where now $\Phi_{t_{i}}=\E^{x_{t_i}}_{t_i}[\E_{t_n}^{x_{t_n}}[\phi_{t_{n+1}}]]=\E^{x_{t_i}}_{t_i}[\phi_{t_{n+1}}]$ for $\hat{R}_{\Phi,i}$ in the above equation. And $\hat{R}_{\Phi,i}$ is still defined in the same way. Notice that $\E[\psi_{t_n}]=\E[\E_{t_n}^{x_{t_n}}[\phi_{t_{n+1}}]]=\E[\phi_{t_{n+1}}]$, we have 
	\begin{align*}
		\hE[\E_{t_n}^{x_{t_n}}[\phi_{t_{n+1}}]]&= \hE[\phi_{n+1}]+\hE[\E_{t_n}^{x_{t_n}}[\phi_{t_{n+1}}]]-\hE[\phi_{n+1}]\\
		&=\hE[\phi_{n+1}]+\hE[\E_{t_n}^{x_{t_n}}[\phi_{t_{n+1}}]-\hE_{t_n}^{x_{t_n}}[\phi_{t_{n+1}}]] \\
		&=\hE[\phi_{n+1}] + \hE[\hat{R}_{\Phi,n}].
	\end{align*}
	Hence we have shown that 
$$	\E[\phi_{t_{n+1}}]=\hE[\phi_{t_{n+1}}] + \sum_{i=0}^{n} \hE[\hat{R}_{\Phi,i}],$$
which completes the proof. 
\end{proof}

In order to carry out our numerical analysis for the BSDE, we need the following boundedness property
\begin{equation}{\label{hate}}
	\hat{\mathbb{E}}[|x_{t_n}|^m] \leq C(|x_0|^m+1),  \ \ m \geq 2,
\end{equation}
where $x_0$ is the initial state of the controlled process. 



In what follows, we prepare lemmas to prove the desired analysis \eqref{hate}, and we shall discuss MLS and RBF separately.

To proceed, we first recall the definition and some facts about MLS. \begin{definition}
	For $x \in \mathbb{R}^d$, the value $s_{\phi, X}(x)$ of the moving least squares interpolant is given by $s_{\phi, X}(x) = p^*(x)$ where $p^*$ is the polynomial that solves the following problem 
	\begin{equation}{\label{mls_initial_def}}
		\min_p \Big \lbrace \sum_{i=1}^N [\phi(x_i)-p(x_i)]^2 w(x,x_i): p \in \pi_l(\R^d) \Big \rbrace,
	\end{equation}
	\end{definition}
\noindent where $\omega(\cdot, \cdot)$ is a positive weight function, $\pi_l$ stands for the space of d-variate polynomials of absolute degree at most $l$. 
\begin{remark}
 It is usually assumed that the weight functions have compact support, so that the function value $\phi(x)$ is only affected by the values around it. 
Therefore, it makes sense to rewrite the above problem in the form: 
	\begin{equation} {\label{mls_reduce}}
		\min_p \Big \lbrace \sum_{i \in  \mathcal{I}(x)} [\phi(x_i)-p(x_i)]^2 w(x,x_i): p(x) \in \pi_l(\R^d) \Big \rbrace,
	\end{equation}
	where $\mathcal{I}(x) = \lbrace j \in \lbrace 1, ..., N\rbrace: ||x-x_j||_2 < r \rbrace$: all the points in the ball $B_{r}(x)$. 
\end{remark}

The following theorem shows that the interpolating function can be written in the Lagrangian form, and it satisfies the polynomial reproduction property. 

\begin{theorem}{\label{lagrange_form}}
	Suppose that for $x \in \mathbb{R}^d$ the set $\lbrace x_j: j \in  \mathcal{I}(x) \rbrace$ is $\pi_l(\R^d)$-unisolvent. In this situation, the problem setup is uniquely solvable and the solution $s_{\phi, X}=p^*(x)$ can be represented as
	\[
	s_{\phi, X} =\sum_{i \in \mathcal{I}(x)} a^*_i(x) \phi(x_i),
	\]
	where the coefficients $a_i^*(x)$ are determined by minimizing the quadratic form 
	\begin{align} \label{mls_lagrange}
		s_{\phi, X}(x) = \sum_{i \in \mathcal{I}(x)} a_i(x)^2 \frac{1}{\Phi (x-x_i)}
	\end{align}
	with the polynomial reproduction constraints
	\begin{equation}
		\sum_{j \in \mathcal{I}(x)} a_i(x) p(x_j) =p(x), \ p \in \pi_l(\R^d).
	\end{equation}
\end{theorem}  

\begin{remark} 
We refer to \cite{wen} for the proof of the above theorem.
\end{remark}  
 The following lemma gives the spatial approximation accuracy of the MLS method, which will be used in the analysis for \eqref{hate}, and we assume the following statements are true:
 \begin{enumerate}
    \item $\Omega$ is compact and satisfies the interior cone condition with angle $\theta \in (0, \pi/2)$ and $r>0$ as defined in Definition 3.6 in \cite{wen}.
 	\item There exist $h_0, C_2$, such that $h<h_0$. $h_0, C_2$ are constants introduced in Theorem 3.14 in \cite{wen} which depend only on $\theta, r , l$, where $l$ is the order of the polynomial to be reproduced. 
 \end{enumerate}
 We remark that $C_2$ will be used in the next Lemma, and the interior cone condition we make on the compact set $\Omega$ enable us to do local polynomial reproduction, which is a key element in carrying out analysis on spatial approximation later. 
 
\begin{lemma}{\label{mls_h2}}
	Define $\Omega^*$ to be the closure of $\cup_{x \in \Omega}B(x,2 C_2 h_0)$. Then there exists a constant $c>0$ that can be computed explicitly such that for all $\phi \in C^{l+1}(\Omega^*)$ and all quasi-uniform $X \subseteq \Omega$ with $h \leq h_0$ the approximation is bounded as follows: 
	\begin{equation}\label{MLS_error}
		|| \phi - s_{\phi,X}  ||_{L_{\infty}} \leq h^{l+1} |\phi|_{C^{l+1}(\Omega^*)}
	\end{equation}
\end{lemma} 


\begin{proof}
We sketch the general idea of the proof for the simplicity of presentation.

Let $p_{\phi} \in \pi_l (\R^d)$, then for some $M>0$
	\begin{align*}
		|  \phi(x)-s_{\phi,X} (x)| & \leq | \phi -p_{\phi}(x)| + |p_{\phi}(x)-\sum_{i \in I(x)} a^*_i (x) \phi(x_i)| \\
		&=|\phi-p_{\phi}(x)| + |\sum_{i \in \mathcal{I}(x)} a^*_i (x)p_{\phi}(x_i)-\sum_{i \in \mathcal{I}(x)} a^*_i (x) \phi(x_i)| \\
		& \leq |\phi-p_{\phi}(x)| + \sum_{i \in \mathcal{I}(x)} |p_{\phi}(x_i)-  \phi(x_i)| |a^*_i (x)| \\ 
		& \leq (1+ \sum_{i \in \mathcal{I}(x)}|a^*_i|)  ||\phi-p_{\phi}||_{L_{\infty}(B(x,Mh))} \leq C ||\phi-p_{\phi}||_{L_{\infty}(B(x,Mh))}
	\end{align*}
	where in the first inequality we used Theorem \ref{lagrange_form}, and the polynomial reproductive property in the equality followed.

	Now, if we pick the polynomial $p_{\phi}$ to be the Taylor expansion of $\phi$ at $x$, that is
	$$p_{\phi}(z) = \sum_{|\alpha| \leq l} \frac{D^{\alpha} \phi(x)}{\alpha!}(z-x)^{\alpha} $$
	where $D^{\alpha}, (z-x)^{\alpha}$ are the multi-index notations for the derivative operator and power, i.e. 
	$$D^{\alpha}:= \frac{\partial^{\alpha_1}}{\partial x_1^{\alpha_1}} \frac{\partial^{\alpha_2}}{\partial x_2^{\alpha_2}}...\frac{\partial^{\alpha_n}}{\partial x_n^{\alpha_n}},\qquad (z-x)^{\alpha}:=(z_1-x_1)^{\alpha_1}...(z_n-x_n)^{\alpha_n}.$$

	Then, for $\xi \in B(x, Ch)$,  we have the following estimate
	$$|\phi(x)-s_{\phi,X}(x)| \leq C \sum_{|\alpha|=l+1} \frac{||D^{\alpha} \phi(\xi)||_{L_{\infty}(B(x,Mh))}}{\alpha!} |\xi-x|^{\alpha} \leq C h^{l+1} |\phi|_{C^{l+1}(\Omega^*)},$$
	where the ball contains all $x_i \in \mathcal{I}(x)$ (the data points) that locally reproduce the polynomial, and $\# \lbrace i | i \in \mathcal{I}(x)\rbrace$ is an integer which is finite and uniformly bounded for all $x$.
	\end{proof}

In order to carry out approximation analysis for the RBF method, one would need a similar spatial interpolation accuracy result similar to Lemma \ref{mls_h2}. To this end, we recall the definition of a radial function and discuss the RBF approximation in what follows. 
\begin{definition}
	A function $\Psi(\R^d) \rightarrow \R$ is called radial if there exists a univariate function $\psi: \R \rightarrow \R $ such that 
	\begin{equation}
		\Psi(x) = \psi(r), \ \ ||x|| =r 
	\end{equation}
	where $||\cdot||$ denotes the Euclidean norm. 
\end{definition}
The following process defines a radial basis interpolating function.

Suppose we have a function $\Psi$ (a radial basis) that is $\cpd$ of order $l$, a data set $\lbrace \left( x_i, \phi({x_i})  \right)\rbrace_i^N$, and we form the following linear system of equations

\begin{equation} {\label{psetup}}
\begin{bmatrix} A & P \\ P^T & 0 \end{bmatrix}
\begin{bmatrix} v \\  z  \end{bmatrix}
 =
 \begin{bmatrix} \phi(X) \\  0  \end{bmatrix},
	\end{equation}
where $P= \lbrace p_j(x_i) \rbrace$ and $p_j \in \pi_{l-1}(\R^d)$, and the set $A=\lbrace \Psi_{i,j} \rbrace$ forms a basis of $\pi_{l-1}(\R^d)$. 
with $v$, $z$ solved through \eqref{psetup}, the RBF interpolation is implemented as following  
\begin{equation}{\label{rbf_f}}
	s_{\phi, X}(x)=\sum_{j=1}^N v_j \Phi(||x-x_j||)+\sum_k z_k p_k(x), \  p_k \in \pi_{l-1}(\R^d).
\end{equation}

Then given that the target function $\phi$ has enough regularity and the fill distance $h$ is small in a bounded domain $\Omega$, one would expect that a similar interpolation accuracy result like Lemma \ref{mls_h2} also holds. In fact, we have the following lemma for RBF approximation.
\begin{lemma}{\label{thinplate}}
	Let $l > m+ d/2$. Suppose that $\Omega \subseteq \R^d$ is open and bounded and satisfies the interior cone condition. Consider the thin-plate splines $\ff_{d,l}$ as $\cpd$ of order $l$. Then the error  between $\phi \in H^l(\Omega)$ and its interpolant $s_{\phi,X}$, which is given in the form of \eqref{rbf_f}, can be bounded by 
	\begin{equation}{}
		|\phi-s_{\phi,X}|_{W^m_p(\Omega)} \leq C h^{l-m-d(1/2-1/p)_+} |\phi|_{BL_{l}(\Omega)}
	\end{equation}
	for $1 \leq p \leq \infty$, and $BL_l(\Omega)$ stands for the Beppo Levi space of order $l$. 
\end{lemma}
\begin{remark}
We refer the readers to \cite{wen} chapter $10$ for the proof.  	
\end{remark}


To see how the above lemma works for example, we take $m=0$, $p=\infty$, $d=2$, and for $\phi \in C^4(\Omega)$, let's take $l =3$. 
	 We have
	\begin{equation}\label{thin_plate_spatial}
		|\phi-s_{\phi,X}|_{L_{\infty}(\Omega)} \leq C h^{2} |\phi|_{BL_{3}(\Omega)}
	\end{equation}
That is, we obtain a second order spatial accuracy approximation $s_{\phi,X}$ for $\phi$ on $\Omega$ in 2D. We point out that this is the set of parameters we take in Section 5 when carrying out numerical experiments for 2D problems. 
We also want to comment that it is sometimes preferable to use the thin-plate splines over the Gaussian kernels as the radial basis, because it's easier to obtain higher approximation accuracy for a wider range of target functions. 

The following theorem shows that the desired estimate \eqref{hate} holds for both MLS and RBF.

\begin{theorem}{\label{lip}}
Under standard assumptions for the MLS interpolation, and pick any RBF interpolation on a bounded domain $\Omega$ with $\oo(h^2)$ spatial accuracy satisfying an estimate like \eqref{thin_plate_spatial},  then for $m \geq 2$, $L \geq 2$, and $h\sim \mathcal{O} (\sqrt{\Delta t})$, the inequality \eqref{hate} holds. 
\end{theorem}
\begin{proof} 
		 By Lemma \ref{mls_h2}, the MLS method provides local polynomial reproduction. 
		Then, by the triangle inequality, we obtain 
		\begin{align}{\label{b1}}
		|I_h |x|^m| \leq |x|^m + ||x|^m -I_h |x|^m| &\leq |x|^m + C\sum_{|\alpha|=2}  \frac{||D^{\alpha}|\xi|^{m}||_{L_{\infty}(B(x,Mh))}}{\alpha !} (\Delta x )^{\alpha}\\
		&\leq |x|^m + C || \ |\xi|^{m-2}||_{L_{\infty}(B(x,Mh))}  |\Delta x|^2
	\end{align}
where for some fixed constant $M$,
\begin{equation}
	 \ \Delta x=y-x,  \ y \in  B(x, Mh),  \ \xi= \gamma x +(1-\gamma) \Delta x , \text{where\ } \gamma \in [0,1].
\end{equation}
	By using the triangle inequality, we have 
	$$ |\xi| \leq |x| + |\Delta x|$$
	Hence, the following estimate holds, 
    \begin{align}
		|x|^m + C || \ |\xi|^{m-2}||_{L_{\infty}(B(x,Mh))}  |\Delta x|^2 & \leq |x|^m + C((|x|+|\Delta x|)^m+1)(\Delta x)^2 \nonumber \\
		 & \leq |x|^m + C(|x|^m + C |\Delta x| (|x|^m+1^m)+1)(\Delta x)^2 \nonumber \\
		 & \leq (1+C|\Delta x|^2)|x|^m + C|\Delta x|^2, \label{est_interm}
	\end{align}
where in each of the inequalities above, we used the bound $|x| \leq max(1, |x|)$, treating $|\Delta x|$ to be small, and the Binomial expansions. We point out that the constant $C$ at each step is only a function of $m$ and $d$ and $M$. 

For the RBF method, under the assumption that the domain $\Omega$ is bounded, we can derive that an inequality similar to \eqref{est_interm} holds. Specifically, we have that the RBF interpolation gives us 
\begin{align}
	|I_h |x|^m| & \leq |x|^m + ||x|^m -I_h |x|^m| \nonumber  \\ 
	& \leq (1+C|\Delta x|^2)|x|^m+C|\Delta x|^2, 
\end{align}
where we have used the result in the estimate \eqref{thin_plate_spatial} for the RBF interpolation.

The above arguments deal with the bound on the interpolation function, and now we shall evaluate the bound of the state process. To proceed, we consider the simulated values of the state
	\begin{equation}
		x_{k,l_1...l_d}=x_k+b(x_k, u(t_n)) \Delta t + \sum_j \sigma_{\cdot, j}(x_k, u(t_n)) \sqrt{\Delta t} \xi^j_{l_j},
	\end{equation}
where $x_k$ is a spatial point for $x_{t_n}$, i.e. $x_{t_n} = x_k$.	
	We let $a_1=x_k+b(x_k, u(t_n)) \Delta t $, and $a_2= \sum_j \sigma_{\cdot, j}(x_k, u(t_n)) \sqrt{\Delta t} \xi^j_{l_j}$ represent the diffusion. 
	Then, we can deduce by using basic inequalities to get
	\begin{align}
		|x_{k,l_1...l_d}|^m=|a_1+a_2|^m \leq |a_1|^m + m |a_1|^{m-2} d|a_1| |a_2| + m^2 d^2 (|a_1|+|a_2|)^{m-2}|a_2|^2.
	\end{align}
	Also, we have the following bounds for $a_1$ and $a_2$: 
	$$|a_1|^m \leq ((1+C\Dt)|x_k|+C \Dt)^m \leq (1+C \Dt)^m|x_k|^m +  ((1+C \Dt)|x_k|+1)^m C \Dt$$
	$$|a_2|^m \leq (1+C \Dt)|x_k|^m + (|x_k|+1)^m C \Dt \leq (1+C \Dt)|x_k|^m +  C \Dt.$$
	By the Lipschtiz property of both $\sigma$ and $b$, we have for small $\Delta t$,
	$$|a_1|+|a_2| \leq C(|x_k|+1)$$
	By choosing $|\Delta x| \sim \oo(h)$ and letting $h$ be of order $\mathcal{O}(\sqrt{\Delta t}) $, we derive a bound for $ \hEkn[|x_{t_{n+1}}|^m]$ as following
	\begin{align*}
		\hEkn[|x_{t_{n+1}}|^m] &=\sum_{l_1=1}^L...\sum_{l_d=1}^L I_{h} |x_{k,l_1...l_d}|^m \prod^d_{j=1} \omega^j_{l_j} \\
		&=(1+C|\Dx|^2)\sum_{l_1=1}^L...\sum_{l_d=1}^L |x_{k,l_1...l_d}|^m \prod^d_{j=1} \omega^j_{l_j}  + C|\Dx|^2 \\
		& \leq (1+C\Dt)((1+C \Dt)|x_k|^m+ C \Dt + \big( (1+C \Dt)|x_k|^{m-1}\\
		&+  C \Dt \big) |a_2| + C(|x_k|^m+1) \Dt) + C \Dt \\
		& \leq (1+C\Dt)((1+C \Dt)|x_k|^m +C \Dt) +C\Dt .
	\end{align*}
	
	Then, by using the estimates on the interpolation operator, the above estimate and the tower property, we obtain the following result: 
	\begin{align*}
		\hEkn[|x_{t_{n+1}}|^m] &=\hat{\mathbb{E}}[ \hat{\mathbb{E}}_{t_n}[|x_{t_{n+1}}|^m]] \\ 
		& \leq (1+C\Dt) ((1+ C\Dt) \hat{\mathbb{E}}[|x_{t_n}|^m]+C \Dt) + C\Dt \\
		& \leq (1+C\Dt)^{n+1} (1+C\Dt)^{n+1} (|x_0|^m+(n+1)C \Dt +(n+1)C \Dt\\
		& \leq C (|x_0|^m+1)
	\end{align*}
	where we also used the fact that $\Delta t \sim 1/N$ and $|\Delta x| \sim \mathcal{O} (\sqrt{\Delta t})$. 
\end{proof}

\subsection{Convergence analysis for the optimal control}
Recall that numerical solutions $\bp_n,\bq_n$ are calculated through \eqref{pde_num} with approximations \eqref{pwpq}. 
The error $\bmu_n= \bp_{t_n}-\bp_n, \bnu_n=\bq_{t_n}-\bq_n$ are of interest, where $(\bp_{t_n},\bq_{t_n})$ are solutions of the FBSDEs system \eqref{SDE}-\eqref{BSDE}.  

Hence, by subtracting (\ref{pde_num}) from (\ref{pde_exact}), we obtain the following system of equations: 
\begin{equation*}\label{sys}
\begin{aligned}
	\bmu_n^k&:=\hEkn[\bmu_{n+1}]+ \Dt \delta f^k_n +{\bR^{x_k}_{p,n}}, \text{ } \bmu_N^k:=\bp^{x_k}_{t_N}-\bp^{x_k}_{N}\\
	\bnu_n^k&:=\frac{1}{\Dt}(\hEkn[\bmu_{n+1}\Delta W^T_{n+1}]+{\bR^{x_k}_{q,n}}),
\end{aligned}	
\end{equation*}
where $\delta f_n^{x_k}=f(x_k,\bp^{x_k}_{t_{n}},\bq^{x_k}_{t_{n}},u(t_n)) - f(x_k,  \bp^{x_k}_n, \bq_n^{x_k}, u(t_n))$, and we recall the error terms ${\bR^{x}_{p,n}}$ and ${\bR^{x}_{q,n}}$ defined in \eqref{condit} - \eqref{pde_num}.

The following lemma gives the bounds for the error terms $\mu_n, \nu_n$ with respect to temporal-spatial truncation errors presented in \eqref{pde_exact}, and we refer readers to \cite{main} (Lemma 4.4) for the proof of the lemma. 
\begin{lemma}{\label{telescope}}
	Assuming that the function $f(x,p,q,u)$ is uniformly Lipschtiz in $x, u$ with respect to $p,q$, then the following estimate holds
	\begin{equation}\label{Error_Truncations}
		\hE[|\bmu_n|^2]+ \Dt \sum_{n=0}^{N-1} \hE[|\bnu_n|^2] \leq C \hE[|\bmu_N|^2] + \frac{C}{\Dt} \sum_{n=0}^{N-1} \hE[|\bR_{p,n}|^2+|\bR_{q,n}|^2].
	\end{equation}
	\end{lemma}

Now we shall provide an estimate for the error term $\frac{1}{\Dt} \sum_{n=0}^{N-1} \hE[|\bR_{p,n}|^2+|\bR_{q,n}|^2]$ on the right hand side of \eqref{Error_Truncations}, which will give us the main convergence results of this paper. 
\begin{lemma}{\label{rpq_estimate}}
	Assuming the Lipschitz continuity on the drift and the diffusion in the forward SDE, and assume that$b(\cdot, \omega), \sigma(\cdot, \omega) \in C^4_b$, $f(\cdot,\cdot,\cdot,\omega) \in C^{2,2,2}_b$ hold uniformly for all $\omega \in \mathcal{C}$, and  $\bp \in C^{1,4}_b$.

	Then the following estimates hold: 
	\begin{equation}
		\frac{1}{\Delta t} \sum_{n=0}^{N-1} \hat{E}[|\bR_{p,n}|^2+|\bR_{q,n}|^2] = \mathcal{O}((\Delta t)^2) + \mathcal{O}(h^4/(\Dt)^2)
	\end{equation}
\end{lemma}
\begin{proof}
To study the error terms $\bR^{x_k}_{p,n}$, $\bR^{x_k}_{q,n}$ (defined in \eqref{pde_exact}), we analyze the error of each term therein. 

	For the MLS case, in \eqref{MLS_error} from Lemma \ref{mls_h2}, one is able to locally reproduce the linear functions, which makes the interpolation error for both $\bR^{x_k}_{I,q,n}$ and $\bR^{x_k}_{I,p,n}$ order $\mathcal{O}(h^2)$, i.e. $|\bR^{x_k}_{I,q,n}|=\mathcal{O}(h^2)$ and $|\bR^{x_k}_{I,p,n}|=\mathcal{O}(h^2)$.
	
	For the RBF case, we known from the estimate \eqref{thin_plate_spatial} that the spatial approximation accuracy is of order $\oo(h^2)$. As a result, the interpolation error for both $\bR^{x_k}_{I,q,n}$ and $\bR^{x_k}_{I,p,n}$ are also of order $\mathcal{O}(h^2)$, i.e.  $|\bR^{x_k}_{I,q,n}|=\mathcal{O}(h^2)$ and $|\bR^{x_k}_{I,p,n}|=\mathcal{O}(h^2)$.
	 	
Let $\bR^{x_k,i}_{E,p,n}$ be the $i$-th component of the vector $\bR^{x_k}_{E,p,n}$, which is the error caused by the Gauss-Hermite interpolation. It is known that in 1D case, for some $\epsilon \in (0,1)$, the following error estimate holds: 
	\begin{equation} \label{hermite_acc}
		|\frac{1}{\sqrt{2\pi} } \int_{\R} f(\xi) e^{-\xi^2/2} d\xi - \sum_{l=1}^L f(\xi_l) \omega_l| \leq \frac{C L^{-r/2}}{\sqrt{2 \pi}} \int_{\R} |f^{(r)}(\xi)e^{-(1-\epsilon) \xi^2/2}| d\xi.
	\end{equation} 	
	Here $L$ is the number of interpolation points. In $d$ dimensions, the tensor scheme requires $L^d$ many points to achieve the same level of accuracy ($L$ in each dimension.) 
	Since $\bp_t, b , \sigma$ are assumed to be four times differentiable in space, by using a similar result of \eqref{hermite_acc} on 
	\[
	f(\xi)=\bp^i_{t_{n+1}}(x_{t_n}+ b (x_{t_n}, u_{t_n}) \Dt+ \sigma(x_{t_n}, u_{t_n})\sqrt{\Dt}\xi)
	\] 
	and the fact
	\begin{equation}
		\sum_{|\alpha|=4}|\frac{\partial^{\alpha} f(\xi)}{\partial^{\alpha_1} \xi_1 ... \partial^{\alpha_d} \xi_d}| \leq C |\sigma(x_k, u(t_n))|^4 (\Dt)^2,
	\end{equation}
	 we obtain the following error bounds for each component: 
	\begin{align*}
		& |\bR^{x_k,i}_{E,p,n}| \leq C |\sigma(x_k, u(t_n))|^4 (\Dt)^2. 
	\end{align*}
Similarly, we obtain the following estimate for  $|\bR^{x_k, ij}_{E,q,n}|$, where $ij$ means the $(i,j)th$ component of $\bR^{x_k}_{E,p,n}$. 
\begin{equation}
	|\bR^{x_k,ij}_{E,q,n}| \leq C |\sigma(x_k, u(t_n))|^4 (\Dt)^{5/2}+ C|\sigma(x_k, u(t_n))|^3 (\Dt)^2.
\end{equation}

	By taking the discrete expectation on both sides, using Theorem \ref{lip} and the assumption that $ \sigma$ is Lipschitz in $x$, we obtain
	 $$\hat{\E}[\sigma(x_{t_n}, u(t_n))^4] \leq \hat{\E}[C*|1+ x_{t_n}|^4] \leq C.$$
	 Thus, by combining components in $\bR_{E,p,n}$, and $\bR_{E,q,n}$,  we conclude that 
	 $$\hat{\E}[|\bR_{E,p,n}|^2]=\Or((\Dt)^4), \ \hat{\E}[|\bR_{E,q,n}|^2]=\Or((\Dt)^4).$$
	 
	To study the integration error $\tilde{\bR}^{x_k}_{p,n}= \E^{x_k}_{t_{n+1}} [\bp_{t_{n+1}}]-\E^{x_k}_{t_{n+1}}[\tilde{\bp}_{t_{n+1}}]$, we notice that the approximation error comes from the  the approximation of the forward SDEs.
	  
	Hence, keeping time $t=t_{n+1}$ fixed and using the multidimensional It\^o formula twice on $\bp ^i$ (the $i$-th component of $\bp$), we have the following:
	\begin{align} 
		\E^{x_k}_{t_n}[\bp^i_{t_{n+1}}] &=\bp^{i,x_k}_{t_{n+1}} +\Dt \mathcal{L}\bp^i (t_{n+1}, x_k) + \int_{t_n}^{t_{n+1}}  \int_{t_n}^{s} \E^{x_k}_{t_n}[\mathcal{L}\mathcal{L} \bp^i (t_{n+1},x_r)]dr ds,\label{quasie}\\
		\E^{x_k}_{t_n}[\tilde{\bp}^i_{t_{n+1}}] &=\bp^{i,x_k}_{t_{n+1}} +\Dt \tilde{\mathcal{L}}\bp^i (t_{n+1}, x_k) + \int_{t_n}^{t_{n+1}}  \int_{t_n}^{s} \E^{x_k}_{t_n}[\tilde{\mathcal{L}}\tilde{\mathcal{L}} \bp^i (t_{n+1},x_r)]dr ds.\label{quasie:2}
	\end{align}
	In the first equation, the state dynamics of $x_k$ has the form
	$$dx_t = b(x_s, u(s))ds + \sigma(x_s, u_s) dW_s,$$
	while the second equation takes the form:
	$$d\tilde{x}_s = b(\tilde{x}_{t_n}, u_{t_n}) ds + \sigma(\tilde{x}_{t_n}, u_{t_n}) dW_s,$$
and the generators $\mathcal{L}$, $\tilde{\mathcal{L}}$ corresponding to $x_t$ and $\tilde{x}_t$ are defined as follows: for any $u \in C^2(\R^d)$,
	\begin{align}
		\mathcal{L} u^k (x_r) &= \sum_{i} b(x_r, u(r))_i^T \partial_i u^k (x_r) + \frac{1}{2} \sum_{i,j} [\sigma(x_s, u_s) \sigma(x_s, u_s)^T]_{ij} \partial_{ij} u^k,\\
		\tilde{\mathcal{L}} u^k (\tilde{x}_r) &= \sum_{i} b(\tilde{x}_{t_n}, u(t_n))_i^T \partial_i u^k (\tilde{x}) + \frac{1}{2} \sum_{i,j}[\sigma(\tilde{x}_{t_n}, u_{t_n}) \sigma(\tilde{x}_{t_n}, u_{t_n})^T]_{ij} \partial_{ij} u^k.
	\end{align}

We can see that first order terms in (\ref{quasie}) and \eqref{quasie:2} are the same, and so the error is of order $\Or ((\Dt)^2)$. And by following exactly the same argument, we conclude that $\tilde{\bR}^{x_k,ij}_{q,n}$ is also of order $\Or ((\Dt)^2)$, and the same error order extends to the vector $\tilde{\bR}^{x_k}_{p,n},\tilde{\bR}^{x_k}_{q,n}$ since it holds component-wise. 
	
	Finally, the discretization error can be easily seen to have error of order $\Or ((\Dt)^2)$ which follows the standard argument of taking It\^o expansion followed by taking the conditional expectation $\E_{t_n}^{x_k}[\cdot]$. We have: 
	\begin{align}
		\bar{\bR}_{p,n}^{x_k} &= \int_{t_n}^{t_{n+1}}  \int_{t_n}^{s} \mathbb{E}^{x_k}_{t_n}[\mathcal{L}^0 \bar{f}]   dt ds, 
			\end{align}
where $\bar{f}(t,x)=f(x,\bp(t,x),\bq(t,x),u(t))$, and we let 
	\begin{align}
		\mathcal{L}^0 \bar{f}_l(t,x)&=\partial_t \bar{f}_l + \sum_i b_i \partial_i \bar{f}_l + \frac{1}{2} \sum_{ij}[\sigma \sigma^T]_{ij} \partial_{ij} \bar{f}_l, \nonumber  \\ 
		\mathcal{L}_k^1 \bar{f}_l(t,x)&= \sum_i \sigma_i \partial_i \bar{f}_l.
	\end{align}
Also, for $\bar{\bR}_{q,n}^{x_k}$, we have
	\begin{align}
		\bar{\bR}_{q,n}^{x_k} &=\int_{t_n}^{t_{n+1}}  \mathbb{E}^{x_k}_{t_n}[  \big( \bar{f} - \bar{f} (t_n, x^k) \big) \Delta W^T_{n+1} ds   - \mathcal{L}^0 \boldsymbol{\xi} ds ] \nonumber   \\
		&= \int_{t_n}^{t_{n+1}}  \mathbb{E}^{x_k}_{t_n}[\big(\int^s_{t_n}  \mathcal{L}^0\bar{f}dt + \mathcal{L}^1 \bar{f}dW_s \big) \Delta W^T_{n+1} - \mathcal{L}^0 \boldsymbol{\xi} ds ] \nonumber \\
		 &=\int_{t_n}^{t_{n+1}}  \int_{t_n}^{s} \mathbb{E}^{x_k}_{t_n}[\mathcal{L}^0 \bar{f} \Delta W^T_{n+1} + \mathcal{L}^1 \bar{f}- \mathcal{L}^0 \boldsymbol{\xi} ] dt ds,
	\end{align}
	where we have used It\^o Isometry and the fact that $\int_t^{t_{n+1}} dW_s $ is independent of rest of the terms in the conditional expectation. Then, we can derive
	$$ |\bar{\bR}_{p,n}^{x_k}|^2 \sim \oo((\Dt)^4), \quad |\bar{\bR}_{q,n}^{x_k}|^2 \sim \oo((\Dt)^4). $$
		
	In conclusion, we have
	$$\hE[|\bR_{I,p,n}|^2]  \sim \Or(h^4), \ \hat{\E}[|\bR_{E,p,n}|^2] \sim \Or((\Dt)^4),  \ \hE[|\tilde{\bR}_{p,n}|^2]\sim \Or((\Dt)^4), \  \hE[|\bar{\bR}_{p,n}|^2] \sim  \Or((\Dt)^4),  $$
	$$\hE[|\bR_{I,q,n}|^2]  \sim \Or(h^4), \ \hat{\E}[|\bR_{E,q,n}|^2] \sim \Or((\Dt)^4),  \ \hE[|\tilde{\bR}_{q,n}|^2]\sim \Or((\Dt)^4), \  \hE[|\bar{\bR}_{q,n}|^2] \sim  \Or((\Dt)^4).  $$
	Then by \eqref{pde_exact}, putting all the terms together, and recall that $\oo(N^{-1}) \sim \oo (\Dt)$  we conclude that 
	$$\frac{1}{\Delta t} \sum_{n=0}^{N-1} \hat{E}[|\bR_{p,n}|^2+|\bR_{q,n}|^2] = \mathcal{O}((\Delta t)^2) + \mathcal{O}(h^4/(\Dt)^2).$$
\end{proof}

With the convergence analysis for the BSDE, we can derive the following error estimates, in which the last assertion $||u^*-u^{N,i}|| \sim \oo(\Delta t)$ gives the desired result for control. We point out quickly that it utilizes the convergence result from Theorem \ref{thm1}. We give a sketch of the proof below. The technical details can be found in \cite{main}, Theorem 4.6.

\begin{theorem}{\label{approximation}}
Under the standard assumptions for $b, \sigma$, and the assumptions in the previous lemmas, the following error estimates hold
	\begin{align}
		\hat{\E}[|\bmu_n|^2] + \Dt \sum_{n=0}^{N-1}\hat{\E}[|\bnu_n|^2]= \mathcal{O}(\Dt^2) + \mathcal{O}(h^4/(\Dt)^2), \label{key_res1} \\
		\sup_i ||J'(u^{N,i})-J_N'(u^{N,i}) ||=\mathcal{O}(\Dt) + \mathcal{O}(h^2/\Dt). \label{eorder}
	\end{align}
\end{theorem}
Then, if $ \oo(h) \sim \oo(\Dt)$, the following relations holds: 
$$\sup_i ||J'(u^{N,i})-J_N'(u^{N,i}) ||= \mathcal{O}(\Dt), \ ||u^*-u^{N,i}|| = \mathcal{O}(\Dt), \ i \rightarrow \infty. $$
\begin{proof}
	\eqref{key_res1} follows immediately from Lemma \ref{telescope} and \ref{rpq_estimate}. 
		By assumption, $u \in U_N$ which is piecewise constant. For convenience, we define $\psi$ to be the following which is assumed to be piecewise $C^{1,4}([0,T),\R)$: 
	\begin{align*}
		\psi_t&=\bp_t b'_u(x_t,u(t)) + tr\big(\bq^\dagger_t \sigma_u'(x_t,u(t))\big)+j'(u(t)), \\
		\psi_n^k&=\bp^k_n b'_u(x_k,u(t_n)) +tr\big((\bq^k_n)^\dagger \sigma_u'(x_k,u(t_n)) \big)+j'(u(t_n)).
	\end{align*}
	Recall that 
	$$J'_N(u)|_{t_n}=\hat{\E}[\psi_n], J'(u)|_t = \E[\psi_t]. $$
	Now we start the following error estimates: 
	\begin{align*}
		||J'(u)-J'_N(u)||^2 		& \leq C \Big ( \sum_{n=0}^{N-1} \int_{t_n}^{t_{n+1}} |J'(u)-J'(u)|_{t_n}|^2 dt  + |J'(u)|_{t_n}-J_N'(u)|_{t_n}|^2 dt \Big ) \\
				& \leq  C \Delta t \sum_{n=0}^{N-1} \int_{t_n}^{t_{n+1}} \int^t_{t_n} |\E[\mathcal{L}(\psi_r)]|^2  dr  dt + C \Delta t \sum_{n=0}^{N-1} |\E[\psi_{t_n}]- \hat{\E}[\psi_n]|^2\\
		& \leq C (\Delta t)^2 +  C \Delta t \sum_{n=0}^{N-1} |\E[\psi_{t_n}]- \hat{\E}[\psi_{t_n}]|^2+ C \Delta t \sum_{n=0}^{N-1} |\hat{\E}[\psi_{t_n}]- \hat{\E}[\psi_{n}]|^2 \\ 
		& \leq C (\Delta t)^2  + C ( h^4/(\Delta t)^2) + C \Delta t \sum_{n=0}^{N-1}  \hat{\E}[|\bmu_n|^2 + |\bnu_n|^2] \\
		&= \oo( (\Delta t)^2)+ \oo((h)^4/(\Delta t)^2 ).
	\end{align*}
Thus \eqref{eorder} holds. Hence, the last statement follows from 	\eqref{eorder}, Theorem \ref{thm1} and the remark that follows.
\end{proof}

\section{Numerical Demonstration}\label{Numerics}
In this section we demonstrate effectiveness and efficiency of our meshfree approximation method for stochastic optimal control problems, and we shall solve the same problem with different choices of dimensions, i.e. $d=2,3,4$. 
For $d=2$ we demonstrate that both the MLS and RBF methods achieve the first order convergence.  
For $d=3$, we compare the classical tensor grid polynomial approximation with the RBF method and show that the latter is much more efficient in terms of both accuracy and computational efficiency. To further demonstrate the performance of our method, we run the example again in the case $d=4$.  
In all the experiments, we will take the tolerance for control to be $\epsilon=10^{-3}$, the number of Monte Carlo samples will be $5\times10^4$.

\vspace{0.25em}

\noindent \textbf{Problem setup. }\textit{
 	The cost functional is given by 
 	\begin{equation}\label{problem_cost}
 		J(u)=\frac{1}{2} \int_0^T \sum^d_{i=1}\E[(y^i-y^*)^2 ]dt +\frac{1}{2} \int_0^T u^2(t) dt , 
 	\end{equation}
 	The forward process is given by 
 	\begin{equation}\label{diff_eqn}
 		d y^i(t) = u(t) y^i(t) dt + \sigma^i y^i(t) dW_t \ , \ \  i=1,\cdots,d.
 	\end{equation}
 	And one needs to find $u^* \in U$ such that 
 	\[
 	J(u^*)=\min_{u\in U} J(u),
 	\]
 	where $u^*$ is the optimal control of this problem. 
}

\noindent \textbf{Exact solutions.}
We will pick two specific functions for $y^*$ and study the numerical solutions of the optimal control $u^*_t$ and compare them to the exact solutions. 
\begin{enumerate}[\textbf{Case }$1$.]
	\item In this case, the function $y^*$  and its corresponding optimal control $u^*$ are given as following 
	\begin{align}
		y^* &= \frac{1}{d}(1 -\frac{(t-T)^2}{\frac{1}{y_0} - Tt +\frac{t^2}{2}}+ \frac{\sum^d_{i=1}e^{\sigma_i^2 t}}{\frac{1}{y_0} - Tt +\frac{t^2}{2}} ), \label{problem_11} \\ 
		u^*&=\frac{T-t}{\frac{1}{x_0} - T t +\frac{t^2}{2}}.  \label{problem_12}
	\end{align}
	\item In this case, the function $y^*$  and its corresponding optimal control $u^*$ are given as following.
	\begin{align}
		y^* &= \frac{1}{d}(\frac{\sum^d_{i=1} e^{\sigma_i^2 t} -(e^{-T}-e^{-t})^2}{\frac{1}{y_0}+1 -e^{-t} -e^{-T}t} -e^{-t} ), \label{problem_21} \\
		u^* &= \frac{e^{-T}-e^{-t}}{\frac{1}{y_0}+1 -e^{-t} -e^{-T}t}. \label{problem_22}
	\end{align}
\end{enumerate}

\vspace{0.2em}

\begin{remark}
In this problem, the control process $u_t$ is deterministic and it is one dimensional. The constant diffusion term is a $d \times d $ diagonal matrix. 
Even though $u_t$ is chosen to be one dimensional, the problem still demonstrates all the difficulties in multidimensional control problems. And this is because the dynamics of $y^1(t),...y^d(t)$ all show up in the the running cost in equation \eqref{problem_cost}. 
\end{remark}

\subsection{d=2}
By Theorem \ref{approximation}, to achieve error of $\oo(\Dt)$ one needs to pick $h \sim \Dt$ that is, the fill distance is on the same scale as $\Dt$.
Hence, we take the number of spatial points to be $N^2$. 

In all of the control plots below, blue dots stand for the discrete control values obtained by using our numerical methods and the red curve is the exact solution. In the error decay log-log plots, blue dots stand for the numerical errors and the red straight line has slope 1. 
\begin{enumerate}[\textbf{Case} $1)$]
	\item We let $y_0=0.5, T=1.0, \sigma_1=0.1, \sigma_2 =0.15$, and numerical results are given in Figure \ref{case1:d2}. Since the approximation is found to be of high accuracy by using only 11-21 temporal points, we study the error decay by using $N=9, 11, 13, 16, 19,21$ temporal points.  
 The corresponding number of spatial points are taken to be equal to $N^2$. We can see from this figure that our methods accurately captured the real optimal control, and the error decay is actually better than first order. Both subplots for control accuracy are done by using $N=21$ points. 
	\item We still let $y_0=0.5, T=1.0, \sigma_1=0.1, \sigma_2 =0.15$, and numerical results are given in Figure \ref{case2:d2}. We pick $N=11, 16, 21, 26, 31,36$ temporal points with $N^2$ spatial points to study the convergence rate, and the control accuracy results are graphed by using $N=36$ temporal points.
 From this figure, one may observe that our methods give very accurate approximations for the optimal control, and both MLS and RBF provide first order convergence rate.  
\end{enumerate}

\begin{figure}[h!]
\center
  \begin{subfigure}[b]{0.4\textwidth}
    \includegraphics[width=\textwidth]{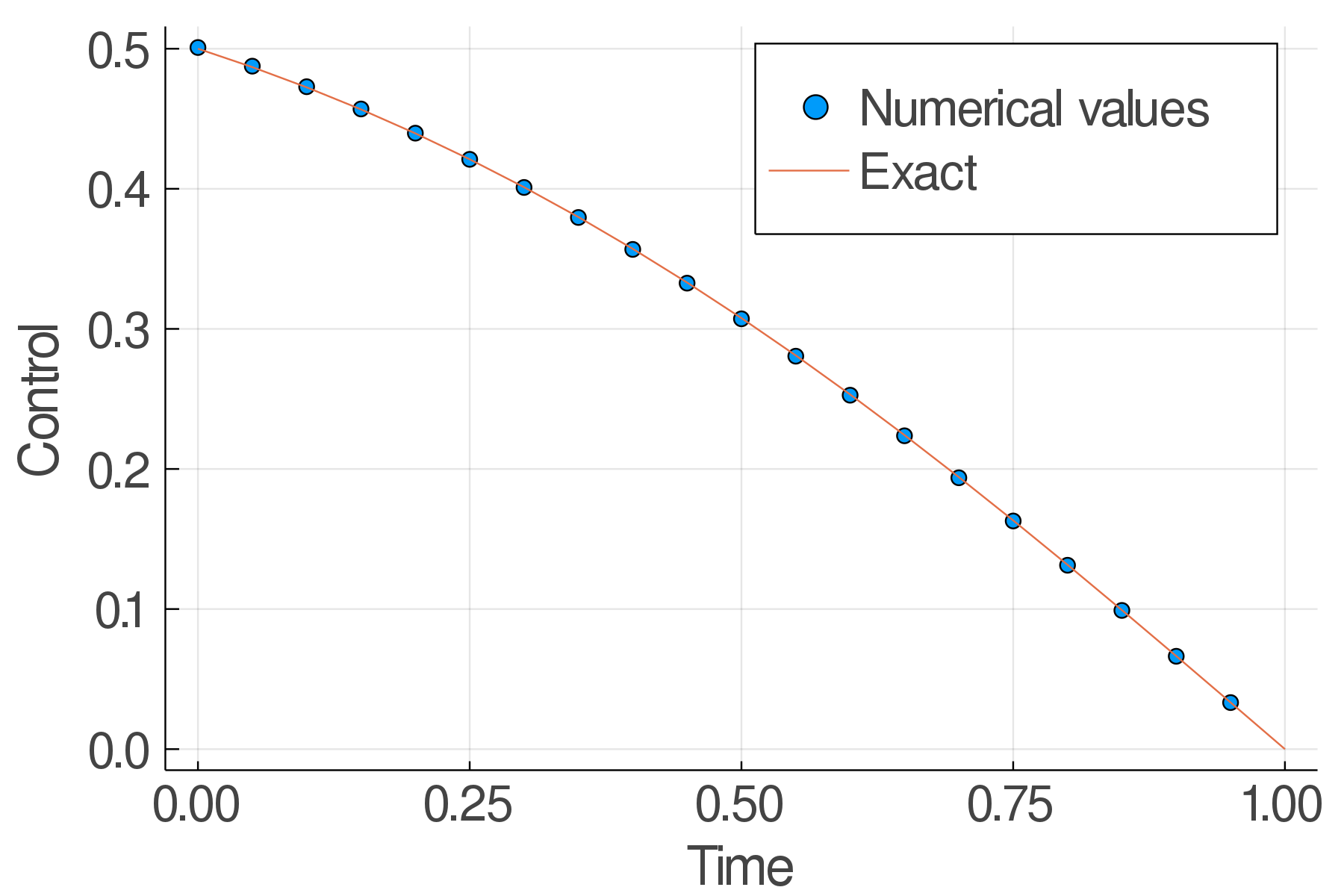}
    \caption{2D Control: Case 1 (RBF)}
    \label{fig:1}
  \end{subfigure}
  \begin{subfigure}[b]{0.4\textwidth}
    \includegraphics[width=\textwidth]{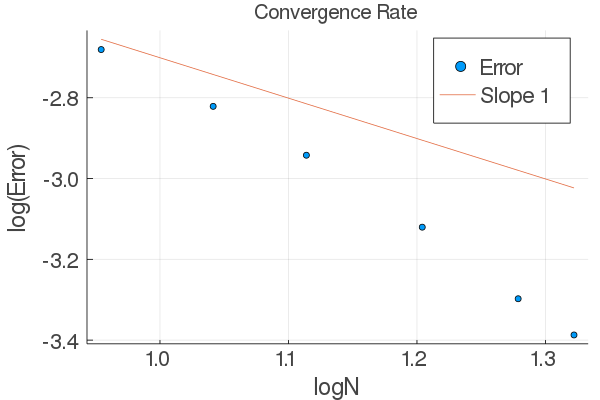}
    \caption{2D Error Decay: Case 1 (RBF)}
    \label{fig:2}
  \end{subfigure}
\center
  \begin{subfigure}[b]{0.4\textwidth}
    \includegraphics[width=\textwidth]{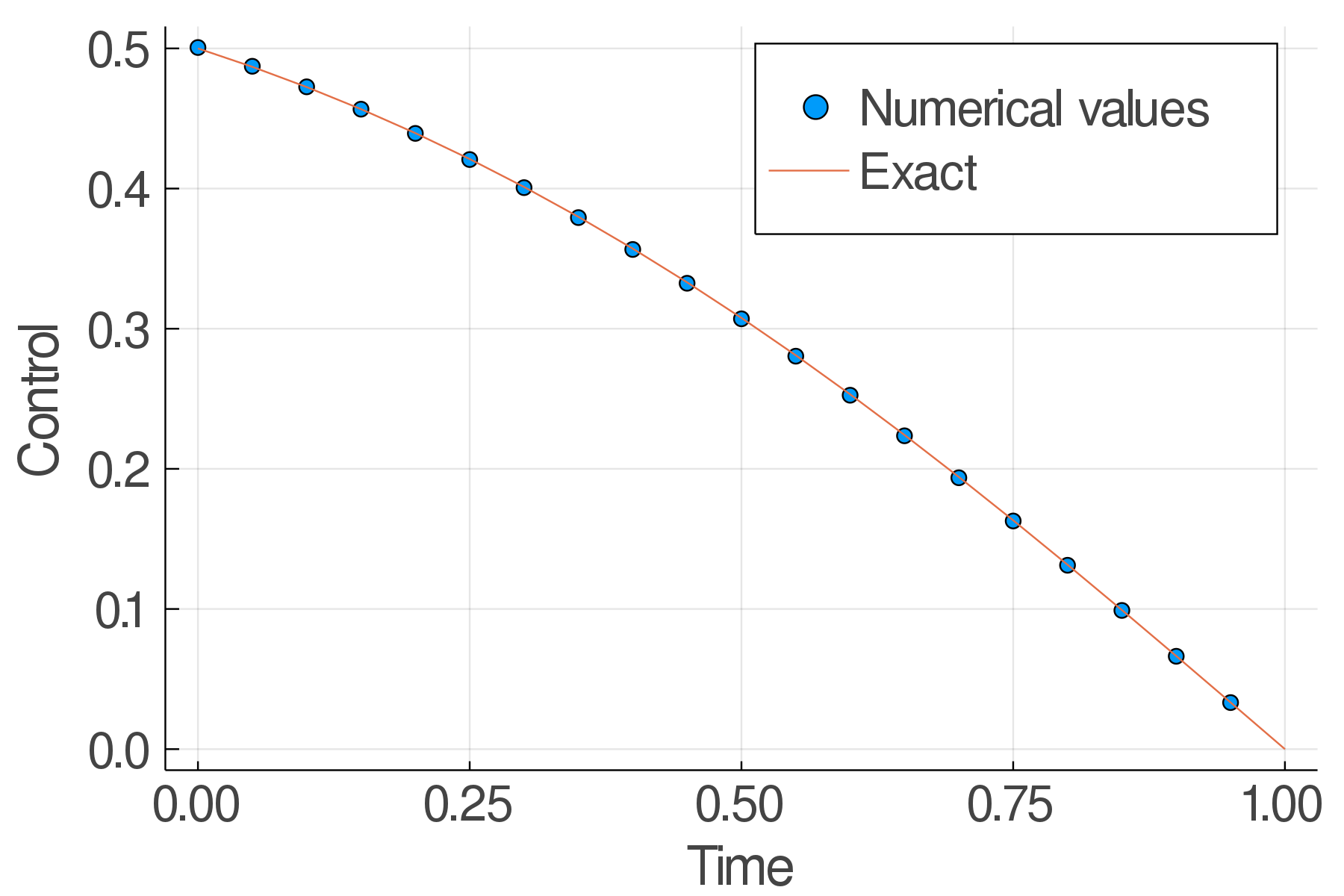}
    \caption{2D Control: Case 1 (MLS)}
    \label{fig:3}
  \end{subfigure}
  \begin{subfigure}[b]{0.4\textwidth}
    \includegraphics[width=\textwidth]{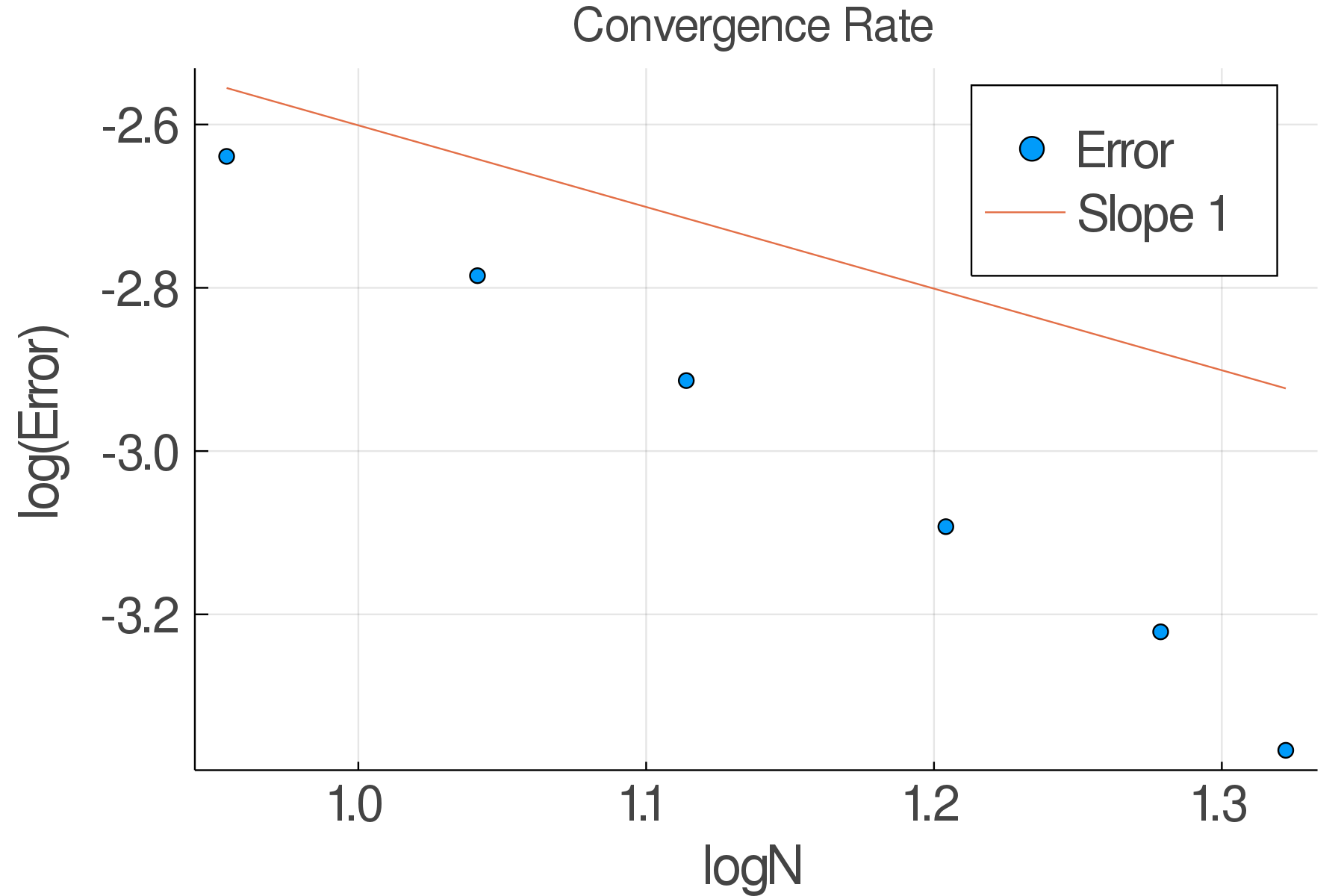}
    \caption{2D Error Decay: Case 1 (MLS)}
    \label{fig:4}
  \end{subfigure}
  \caption{Numerical results for case 1, $d=2$}\label{case1:d2}
\end{figure}

\begin{figure}[h!]
\center
  \begin{subfigure}[b]{0.4\textwidth}
    \includegraphics[width=\textwidth]{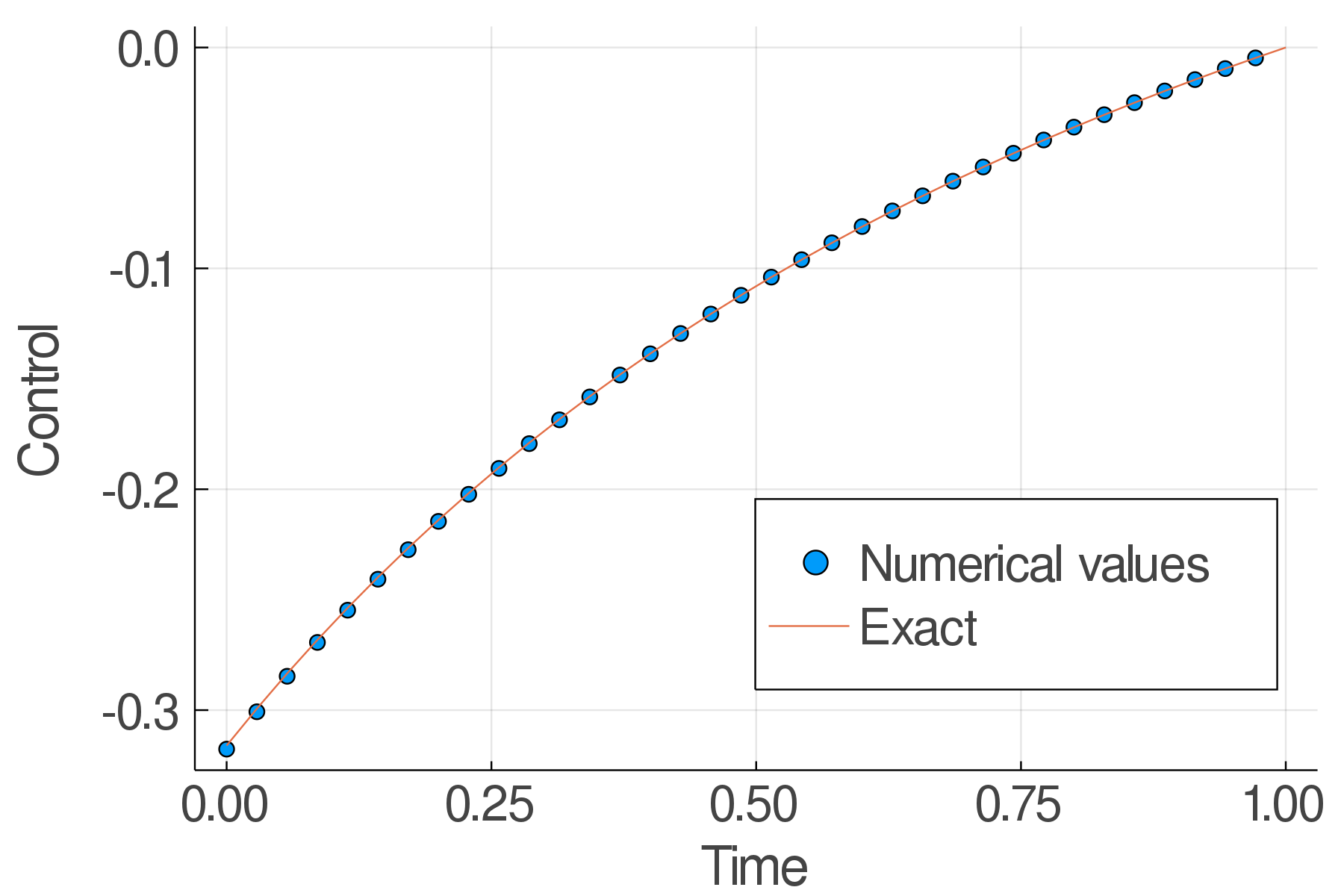}
    \caption{2D Control: Case 2 (RBF)}
    \label{fig:21}
  \end{subfigure}
  \begin{subfigure}[b]{0.4\textwidth}
    \includegraphics[width=\textwidth]{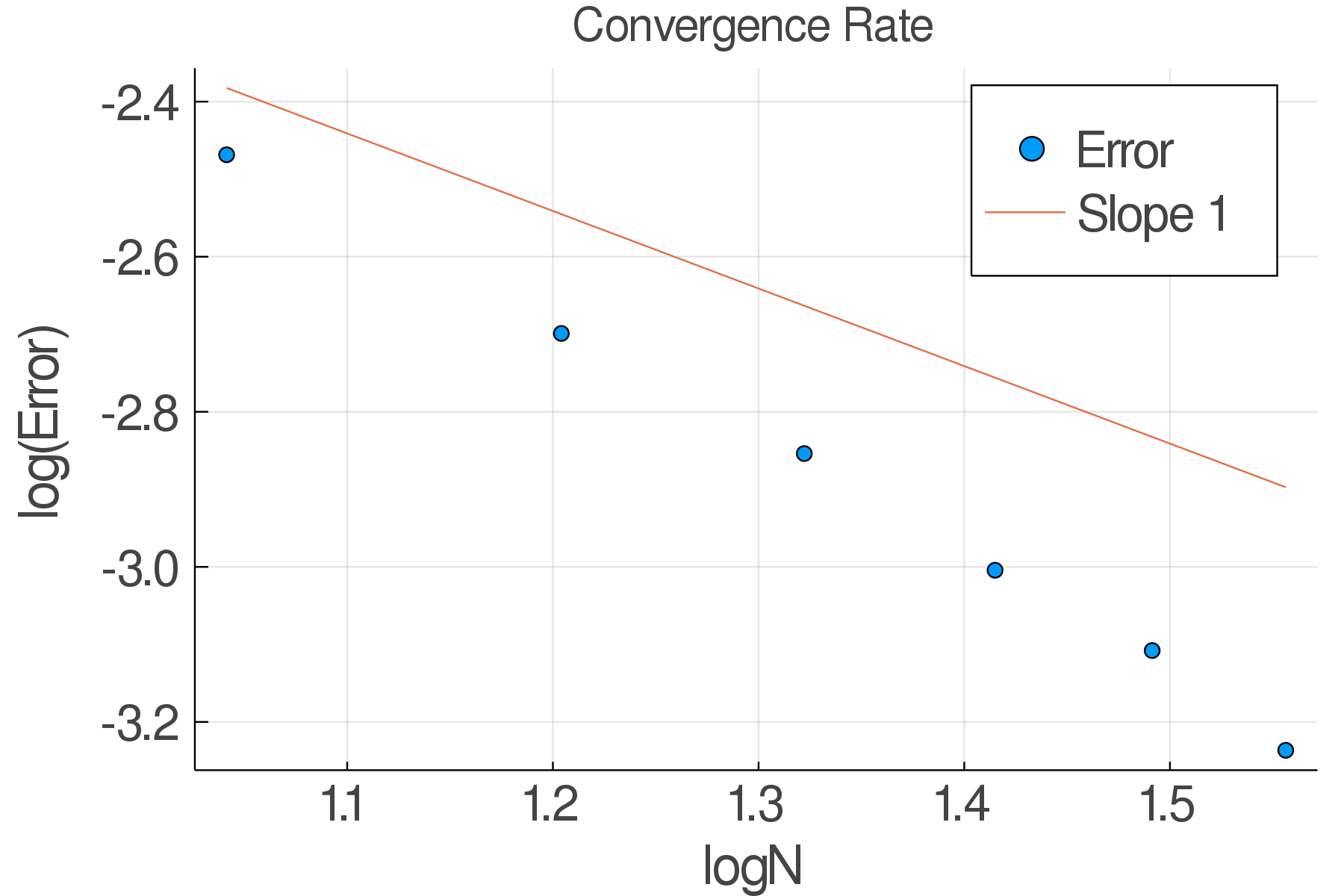}
    \caption{2D Error Decay: Case 2 (RBF)}
    \label{fig:22}
  \end{subfigure}
\center
  \begin{subfigure}[b]{0.4\textwidth}
    \includegraphics[width=\textwidth]{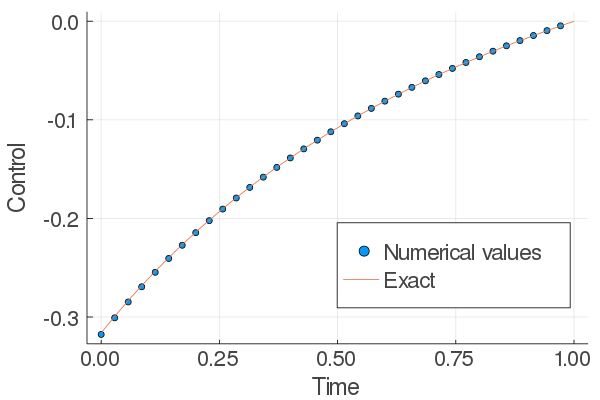}
    \caption{2D Control:  Case 2 (MLS)}
    \label{fig:23}
  \end{subfigure}
  \begin{subfigure}[b]{0.4\textwidth}
    \includegraphics[width=\textwidth]{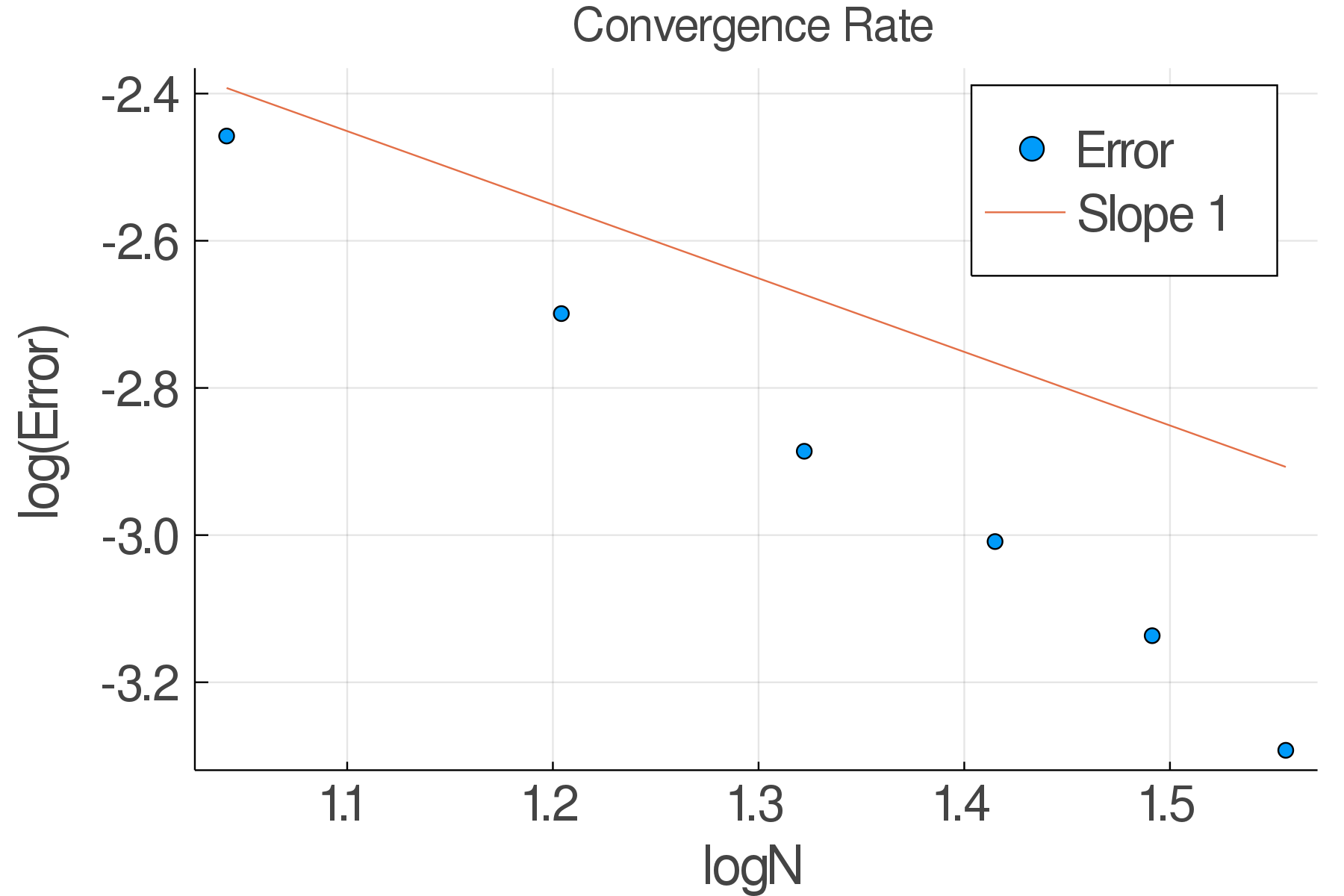}
    \caption{2D Error Decay:  Case 2 (MLS)}
    \label{fig:24}
  \end{subfigure}
  \caption{Numerical results for case 2, $d=2$}\label{case2:d2}
\end{figure}

\subsection{d=3}
In this 3D experiment, we show that the meshfree approximation approach is more efficient than the classic polynomial interpolation method, and we choose the RBF as our meshfree interpolation method. 
For both cases in this numerical demonstration, we set $\sigma_1=0.1, \sigma_2=0.15, \sigma_3=0.2$ and $y_0 =0.5$. Also, for the tensor product polynomial interpolation method, we choose $9$ spatial points in each dimension ($729$ points in total), and we use $216$ Halton points for the RBF method. 

In Figure \ref{d3}, we present the estimation performance for optimal controls, where the cross marks are the numerical results obtained by using the polynomial interpolation method (trilinear was used) and dots stand for the RBF method. In both cases the runtime of the trilinear method is roughly $420$ seconds which is larger than the $130$ seconds runtime of the RBF method. 
We can observe that the RBF method outperforms polynomial approximation in case 1, and case 2 is not very sensitive to the method that we used. 

To make the advantage of the meshfree method more pronounced, we will examine those two cases in dimension four. 

\begin{figure}[h!]
\center
  \begin{subfigure}[b]{0.4\textwidth}
    \includegraphics[width=\textwidth]{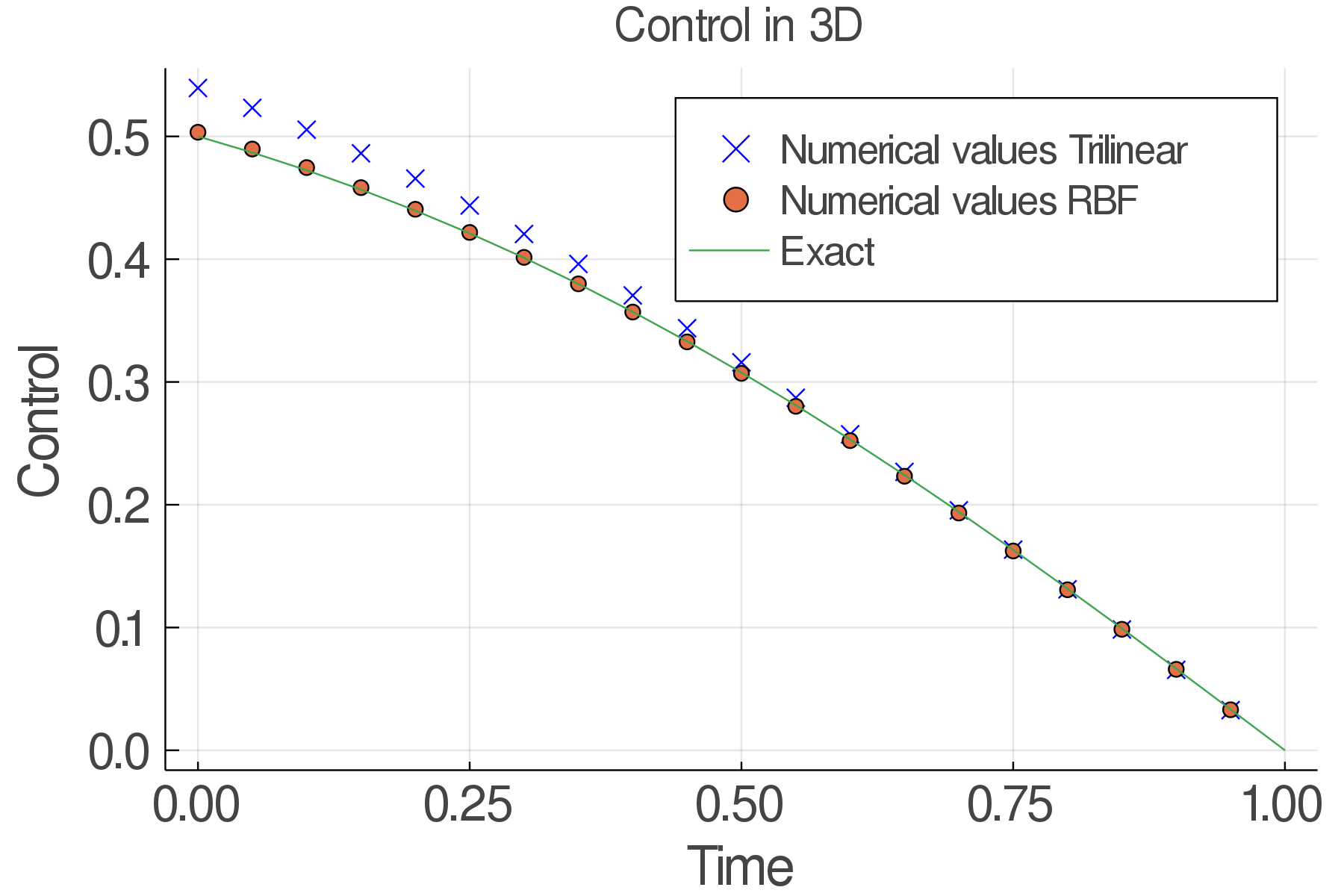}
    \caption{3D Control: Case 1}
    \label{fig:1}
  \end{subfigure}
  \begin{subfigure}[b]{0.4\textwidth}
    \includegraphics[width=\textwidth]{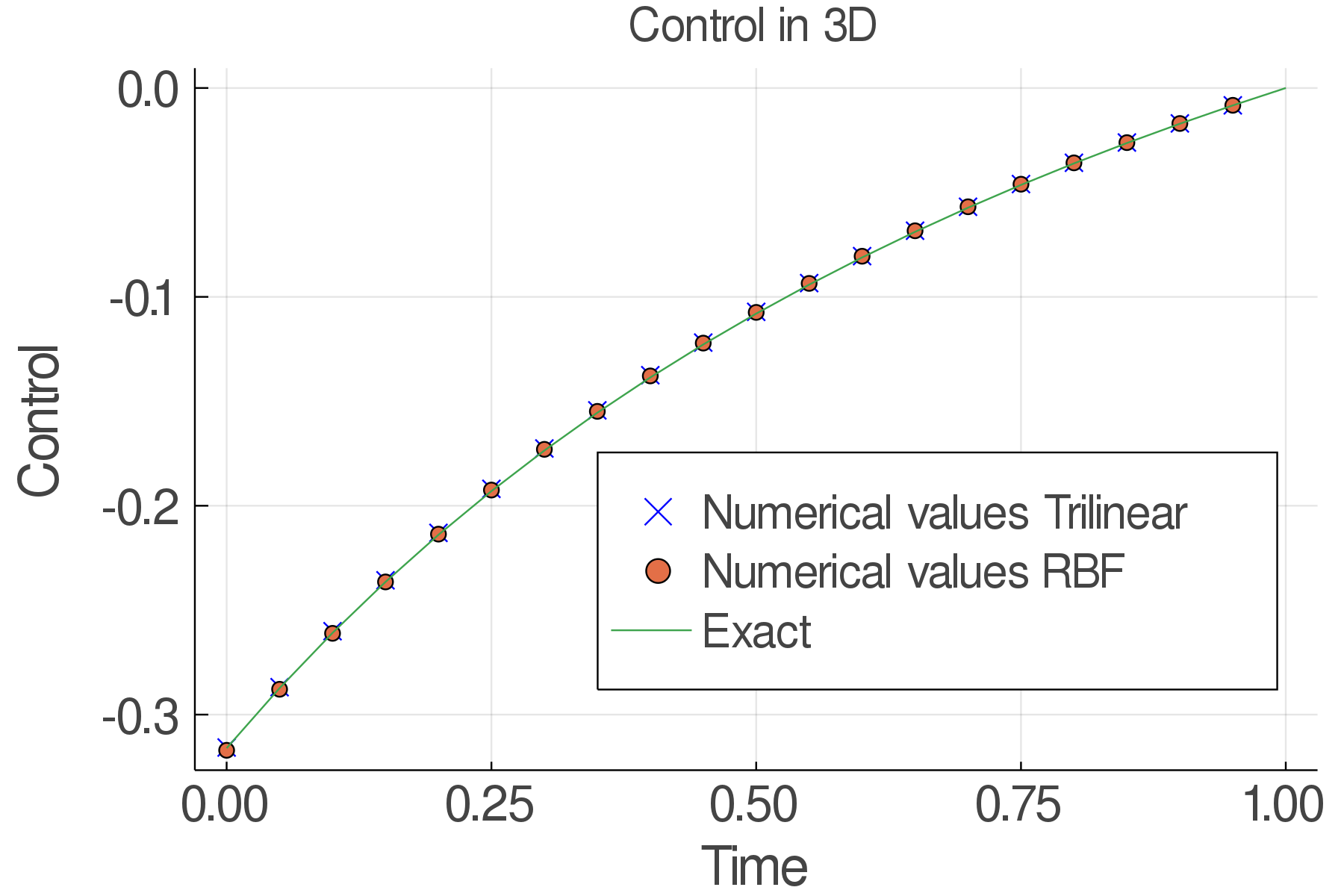}
    \caption{3D Control: Case 2}
    \label{fig:2}
  \end{subfigure}
    \caption{Numerical results for $d=3$}\label{d3}
\end{figure}

\subsection{d =4 }
In this 4D experiment, we set $\sigma_1=0.1, \sigma_2=0.15, \sigma_3=0.2, \sigma_4=0.25$ and $y_0=0.5$ in both cases.
\begin{figure}[h!]
\center
  \begin{subfigure}[b]{0.4\textwidth}
    \includegraphics[width=\textwidth]{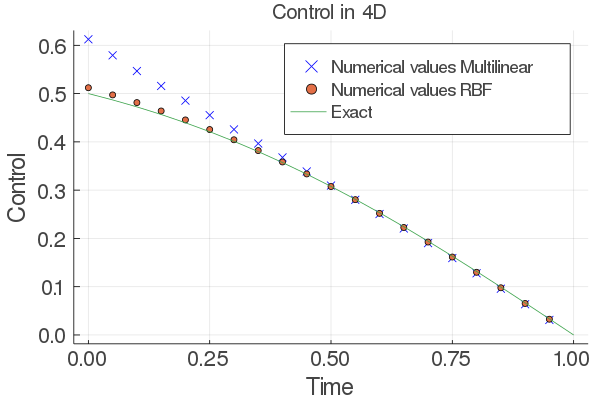}
    \caption{4D Control: Case 1 }
    \label{fig:1}
  \end{subfigure}
  \begin{subfigure}[b]{0.4\textwidth}
    \includegraphics[width=\textwidth]{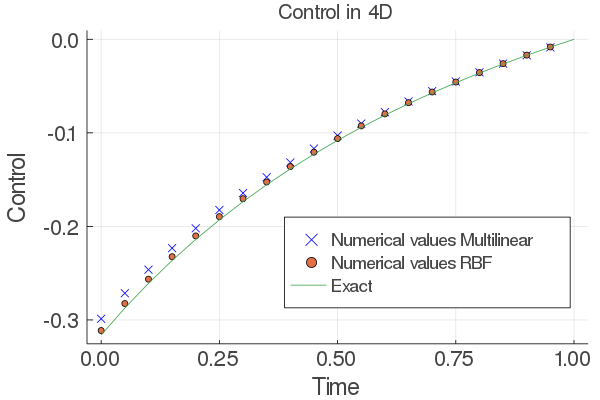}
    \caption{4D Control: Case 2}
    \label{fig:2}
  \end{subfigure}
  \caption{Numerical results for $d=4$}\label{d4}
\end{figure}
For the RBF method, we take $N=21$ temporal points and $216$ spatial points. The runtime for both cases are around $1100$ seconds. For the polynomial approximation method, we also take $N=21$, and we choose $5$ spatial points in each dimension, i.e. altogether $625$ spatial points. The runtime for both cases turns out to be around $2800$ seconds. 
In Figure \ref{d4}, we present the estimation performance for optimal controls. In both subplots, the blue crosses are the numerical results obtained by using the polynomial interpolation method, the red dots are estimates obtained by using the RBF method, and the green curve is the exact optimal control. From this figure, we can see that the RBF constantly outperforms the polynomial interpolation in accuracy -- especially in case 1, and the computational time for the RBF method is less than half of what the polynomial interpolation method costs. This indicates both accuracy and efficiency of our meshfree approximation method.

\end{document}